\documentclass[12pt,  reqno]{amsart}

\usepackage{amsfonts}
\usepackage{amscd}
\usepackage{amsmath,amssymb,amsthm} 
\usepackage{mathrsfs}
\usepackage{enumerate}
\usepackage{float} 
\usepackage[pdftex]{graphicx}
\allowdisplaybreaks

\usepackage{hyperref}

\usepackage{color}

\theoremstyle{plain}
\newtheorem{thm}{Theorem}[section]
\newtheorem{lem}[thm]{Lemma}
\newtheorem{prop}[thm]{Proposition}
\newtheorem{cor}[thm]{Corollary}

\theoremstyle{definition}

\newtheorem*{thm*}{Theorem}

\theoremstyle{remark}
\newtheorem{rem}{Remark}[section]

\numberwithin{equation}{section}

\newcommand{\He}{\mathbb{H}}
\newcommand{\C}{\mathbb{C}}
\newcommand{\R}{\mathbb{R}}

\usepackage{color}

\title[ Chernoff and Ingham on the Heisenberg group]
{On theorems of Chernoff and Ingham\\
 on the Heisenberg group}

\author[Bagchi, Ganguly, Sarkar and Thangavelu]
{Sayan Bagchi, Pritam Ganguly, Jayanta Sarkar\\ and Sundaram Thangavelu}

\address[S. Bagchi]{Department of Mathematics and Statistics\\
	Indian Institute of Science Education and Research Kolkata\\
	Mohanpur-741246, Nadia, West Bengal, India.} 
\email{sayansamrat@gmail.com}

\address[P. Ganguly, S. Thangavelu]{Department of Mathematics, Indian Institute of Science,  Bangalore-560 012, India.}
\email{pritamg@iisc.ac.in, veluma@iisc.ac.in}

\address[J. Sarkar]{Stat-Math Unit, Indian Statistical Institute, 203, B.T. Road, Kolkata-700108, India}
\email{jayantasarkarmath@gmail.com}

\keywords{Heisenberg group, sublaplacian, quasi-analyticity, sobolev spaces, Chernoff's theorem, Ingham's theorem.}
\subjclass[2010]{Primary: 43A80. Secondary: 22E25, 33C45, 26E10, 46E35.}

\begin{document}
\begin{abstract}  We prove an analogue of Chernoff's theorem for the sublaplacian on the Heisenberg group and use it to prove a version of Ingham's theorem for the Fourier transform on the same group.

\end{abstract}
\maketitle

\section{Introduction} Roughly speaking, the uncertainty principle for the Fourier transform on $ \R^n $ says that a function $ f $ and its Fourier transform $\widehat{f} $ cannot both have rapid decay. Several manifestations of this principle are known: Heisenberg-Pauli-Weyl inequality, Paley-Wiener theorem, Hardy's uncertainty principle are some of the most well known. But there are lesser known results such as theorems of Ingham and  Levinson. The best decay a non trivial function can have is vanishing identically outside a compact set and for such functions it is well known that their Fourier transforms extend to $ \C^n $ as entire functions and hence cannot vanish on any open set. For  any such function of compact support, its Fourier transform cannot have any exponential decay for a similar reason: if $ |\widehat{f}(\xi)| \leq C e^{-a|\xi|} $ for some $ a > 0 $, then it follows that $ f $ extends to a tube domain in $ \C^n $ as a holomorphic function and hence it cannot have compact support. So it is natural to ask the question: what is the best possible decay that is allowed of a function of compact support? An interesting answer to this question is provided by the following theorem of Ingham \cite{I}.

\begin{thm}[Ingham]\label{ingh}
Let $ \Theta(y) $ be a nonnegative even function on $\R$ such that $ \Theta(y) $ decreases to zero when $ y \rightarrow \infty.$ There exists a nonzero continuous function $ f $ on $ \R,$ equal to zero outside an interval $ (-a,a) $ having Fourier transform $ \widehat{f} $ satisfying the estimate $ |\widehat{f}(y)|\leq C e^{-|y|\Theta(y)} $ if and only if $ \Theta $ satisfies $ \int_1^\infty \Theta(t) t^{-1} dt <\infty.$
\end{thm}

This theorem of Ingham and its close relatives   Paley -Wiener (\cite{PW1, PW2})  and Levinson (\cite{Levinson}) theorems have received considerable attention in recent years. In \cite{BRS} Bhowmik et al proved analogues of the above theorem for $ \R^n,$ the $n$-dimensional torus $ \mathbb{T}^n $  and step two nilpotent Lie groups. See also the recent work of Bowmik-Pusti-Ray \cite{BPR} for a version of Ingham's theorem for the Fourier transform on Riemannian symmetric spaces of non-compact type. As we are interested in Ingham's theorem on the Heisenberg group, let us recall the result proved in \cite{BRS}. Let $ \He^n = \C^n \times \R $ be the Heisenberg group. For an integrable function $ f $ on $ \He^n $ let $ \widehat{f}(\lambda) $ be the operator valued Fourier transform of $ f $ indexed by non-zero real $ \lambda.$  Measuring the decay of the Fourier transform in terms of the Hilbert-Schmidt operator norm $ \|\widehat{f}(\lambda)\|_{HS} $ Bhowmik et al have proved the following result.

\begin{thm}[Bhowmik-Ray-Sen]\label{BRS}
Let $ \Theta(\lambda) $ be a nonnegative even function on $\R$ such that $ \Theta(\lambda) $ decreases to zero when $ \lambda \rightarrow \infty.$ There exists a nonzero, compactly supported  continuous function $ f $ on $ \He^n,$  whose Fourier transform satisfies the estimate $ \|\widehat{f}(\lambda)\|_{HS} \leq C |\lambda|^{n/2}e^{-|\lambda| \Theta(\lambda)} $ if  the integral  $ \int_1^\infty \Theta(t) t^{-1} dt <\infty.$ On the other hand, if the above estimate is valid for a function $ f $ and the integral $ \int_1^\infty \Theta(t) t^{-1} dt $ diverges, then the vanishing of $ f $ on any set of the form $ \{z\in \C^n: |z|< \delta \} \times \R $ forces $ f $ to be identically zero.
\end{thm}

As the Fourier transform on the Heisenberg group is operator valued, it is natural to measure the decay of $ \widehat{f}(\lambda) $ by comparing it with the Hermite semigroup $ e^{-aH(\lambda)} $ generated by $ H(\lambda)= -\Delta_{\mathbb{R}^n}+\lambda^2 |x|^2.$ In this connection, let us recall the following two versions of Hardy's uncertainty principle. Let $ p_a(z,t) $ stand for the heat kernel associated to the sublaplacian $ \mathcal{L} $ on the Heisenberg group whose Fourier transform turns out to be  the Hermite semigroup $ e^{-aH(\lambda)} .$ The version in which one measures the decay of $ \widehat{f}(\lambda) $ in terms of its Hilbert-Schmidt operator norm reads as follows. If
\begin{equation}
|f(z,t)| \leq C e^{-a(|z|^2+t^2)} ,\,\,  \|\widehat{f}(\lambda)\|_{HS} \leq C e^{-b \lambda^2 } 
\end{equation}
then $ f = 0 $ whenever $ ab > 1/4.$ This is essentially a theorem in the $ t$-vairable and can be easily deduced from Hardy's  theorem on $ \R $, see Theorem 2.9.1 in \cite{TH3}. Compare this with the following version, Theorem 2.9.2 in \cite{TH3}. If 
\begin{equation}
|f(z,t)| \leq C p_a(z,t) ,\,\,   \widehat{f}(\lambda)^\ast \widehat{f}(\lambda) \leq C e^{-2bH(\lambda)} 
\end{equation}
then $ f = 0 $ whenever $ a < b.$ This latter version is the exact analogue of Hardy's theorem for the Heisenberg group, which we can view not merely as an uncertainty principle but also as a characterisation of the heat kernel. Hardy's theorem in the context of semi-simple Lie groups and non-compact Riemannian symmetric spaces are also to be viewed in this perspective.

We remark that the  Hermite semigroup has been used to measure the decay  of the Fourier transform in connection with the heat kernel transform \cite{KTX}, Pfannschmidt's theorem \cite{TH4} and the extension problem for  the sublaplacian \cite{RT} on the Heisenberg group. In connection with the study of Poisson integrals, it has been noted in \cite{ TH5} that when the Fourier transform of $ f $ satisfies an estimate of the form $\widehat{f}(\lambda)^\ast \widehat{f}(\lambda) \leq C e^{-a \sqrt{H(\lambda)}} ,$ then the function extends to a tube domain in the complexification of $ \He^n $ as a holomorphic function and hence the vanishing of $ f $ on an open set forces it to vanish identically. It is therefore natural to ask if the same conclusion can be arrived at by replacing the constant $ a $ in the above estimate by an operator $ \Theta(\sqrt{H(\lambda)}) $ for a function $ \Theta $ decreasing to zero at infinity. Our investigations have led us to the following analogue of Ingham's theorem for the Fourier transform on $ \He^n.$ 

\begin{thm}\label{ingh-hei}
Let $ \Theta(\lambda) $ be a nonnegative even function on $\R$ such that $ \Theta(\lambda) $ decreases to zero when $ \lambda \rightarrow \infty $  and satisfies the condition $ \int_1^\infty \Theta(t) t^{-1} dt <\infty.$  Then there exists a nonzero compactly supported continuous function $ f $ on $ \He^n$ whose  Fourier transform $ \widehat{f} $ satisfies  the estimate 
\begin{equation} \widehat{f}(\lambda)^\ast \widehat{f}(\lambda) \leq C  e^{-2\Theta(\sqrt{H(\lambda)})\sqrt{H(\lambda)}} .\end{equation}
 Conversely, if there exists a nontrivial $f $ satisfying (1.3) and the extra assumption   $ f(z,t) = f( |(z,t)| |z|^{-1}z, 0) $ where $ |(z,t)| = (|z|^4+ t^2)^{1/4} $ is the Koranyi norm on $ \He^n$  which vanishes on  a neighbourhood of $ 0$ then it is necessary that   $ \int_1^\infty \Theta(t) t^{-1} dt <\infty .$ 
\end{thm}

We believe  that the above result is true without the extra assumption on $ f .$  As the proof requires a general version of Chernoff's theorem for the sublaplacian which is yet to be proved (see Theorem \ref{CH} below), we only have the above result at present. However, the class of functions to which the above theorem applies is relatively large. To see this, consider the Heisenberg coordinates $ (\rho, \omega, \theta) $ of a point $ (z,t) \in \He^n $ defined by $   z = r \omega, \, r >0, \omega \in S^{2n-1},  t+ir^2 = \rho^2 e^{i\theta}, 0 \leq \theta \leq \pi$  so that $ f(z,t) = f(\rho \omega \sqrt{\sin \theta}, \rho^2 \cos \theta). $ The functions $ f $ satisfying the extra assumption in Theorem \ref{ingh-hei} are precisely those which are independent of $ \theta $ in the Heisenberg variables. Any function $ g $ on  $ \C^n $ gives rise to such a function on $ \He^n $ by the prescription $ f(z,t) = g(\rho \omega).$

Theorem \ref{ingh} was proved in \cite{I} by Ingham by making use of  Denjoy-Carleman theorem on quasi-analytic functions. In \cite{BRS} the authors have used Radon transform and a several variable extension of Denjoy-Carleman theorem due to Bochner and Taylor \cite{BT} in order to prove the $n$-dimensional version of Theorem \ref{ingh}. An $ L^2 $  variant of the result of Bochner-Taylor which was proved by Chernoff in \cite{CH2} has turned out to be very useful in establishing Ingham type theorems.  

\begin{thm} \cite[Chernoff]{CH2}
	Let $f$ be a smooth function on $\mathbb{R}^n.$ Assume that $\Delta^mf\in L^2(\mathbb{R}^n)$ for all $m\in \mathbb{N}$ and $\sum_{m=1}^{\infty}\|\Delta_{\mathbb{R}^n}^mf\|_2^{-\frac{1}{2m}}=\infty.$ If $f$ and all its partial derivatives   vanish at $0$, then $f$ is identically zero.
\end{thm} 
This theorem shows how partial differential operators generate the class of quasi-analytic functions. Recently,  Bhowmik-Pusti-Ray \cite{BPR} have established an analogue of Chernoff's theorem for the Laplace-Beltrami operator on non-compact Riemannian symmetric spaces and use the same in proving a version of Ingham's theorem for the Helgason Fourier transform.
It  is therefore natural to look for an analogue of this result for sublaplacian on the Heisenberg group. In this paper, we prove the following result.
\begin{thm} \label{CH}
Let $ \mathcal{L} $ be the sublaplacian on the Heisenberg group and  let $f$ be a smooth function on $\mathbb{H}^n$ such that $ f(\rho \, \omega \sqrt{\sin \theta}, \rho^2 \cos \theta) = f(\rho \, \omega , 0) $ in the Heisenberg coordinates. Assume that $\mathcal{L}^mf\in L^2(\mathbb{H}^n)$ for all $m\in \mathbb{N}$ and $\sum_{m=1}^{\infty}\|\mathcal{L}^mf\|_2^{-\frac{1}{2m}}=\infty$. If $f$ and  all its partial derivatives vanish at $0$, then $f$ is identically zero.
\end{thm}

An immediate corollary of this theorem  is the following, which can be seen as an $L^2$ version of the classical Denjoy-Carleman theorem on the Heisenberg group using iterates of sublaplacian.
\begin{cor}
	Let $\{M_k\}_k$ be a log convex sequence. Define $\mathcal{C}(\{M_k\}_k,\mathcal{L},\mathbb{H}^n)$ to be the class of all smooth functions $f$ on $\mathbb{H}^n$ satisfying the condition $ f(\rho \, \omega \sqrt{\sin \theta}, \rho^2 \cos \theta) = f(\rho \, \omega , 0) $ such that $\mathcal{L}^mf\in L^2(\mathbb{H}^n)$ for all $  k\in \mathbb{N}$ and $\|\mathcal{L}^kf\|_2\leq M_k\lambda^k$ for some constant $\lambda$ (may depend on $f$). Suppose that $\sum_{k=1}^{\infty}M_k^{-\frac{1}{2k}}=\infty.$ Then every member of that class is quasi-analytic.
\end{cor}

We conclude the introduction by briefly describing the organisation of the paper. After recalling the required preliminaries regarding harmonic analysis on Heisenberg group  in Section 2 we prove an analogue of Chernoff's theorem for the sublaplacian (Theorem \ref{CH})  in Section 3. In section 4, we use this version of Chernoff's theorem to prove Ingham's theorem on the Heisenberg group i.e., Theorem \ref{ingh-hei}.   

\section{Preliminaries} 
  In this section, we collect the results which are necessary for the study of uncertainty principles for the Fourier transform on the Heisenberg group. We refer the reader to the  two classical books Folland \cite{F} and Taylor \cite{Taylor} for the preliminaries of harmonic analysis on the Heisenberg group. However, we will be closely following the notations of  the books of Thangavelu  \cite{TH2} and \cite{TH3}.  
\subsection{ Heisenberg group and Fourier transform}
Let $\mathbb{H}^n:=\mathbb{C}^n\times\mathbb{R}$ denote the $(2n+1)$-  Heisenberg group equipped with the group law 
$$(z, t).(w, s):=\big(z+w, t+s+\frac{1}{2}\Im(z.\bar{w})\big),\ \forall (z,t),(w,s)\in \mathbb{H}^n.$$ This is a step two nilpotent Lie group where the Lebesgue measure $dzdt$ on $\mathbb{C}^n\times\mathbb{R}$ serves as the Haar measure. The representation theory of $\mathbb{H}^n$ is well-studied in the literature. In order to define Fourier transform, we use the Schr\"odinger representations as described below.  

 For each non zero real number $ \lambda $ we have an infinite dimensional representation $ \pi_\lambda $ realised on the Hilbert space $ L^2( \R^n).$ These are explicitly given by
$$ \pi_\lambda(z,t) \varphi(\xi) = e^{i\lambda t} e^{i(x \cdot \xi+ \frac{1}{2}x \cdot y)}\varphi(\xi+y),\,\,\,$$
where $ z = x+iy $ and $ \varphi \in L^2(\R^n).$ These representations are known to be  unitary and irreducible. Moreover, by a theorem of Stone and Von-Neumann, (see e.g., \cite{F})  upto unitary equivalence these account for all the infinite dimensional irreducible unitary representations of $ \mathbb{H}^n $ which act as $e^{i\lambda t}I$ on the center. Also there is another class of finite dimensional irreducible representations. As they  do not contribute to the Plancherel measure  we will not describe them here.

The Fourier transform of a function $ f \in L^1(\mathbb{H}^n) $ is the operator valued function obtained by integrating $ f $ against $ \pi_\lambda$:
$$ \hat{f}(\lambda) = \int_{\mathbb{H}^n} f(z,t) \pi_\lambda(z,t)  dz dt .$$  Note that $ \hat{f}(\lambda) $ is a bounded linear operator on $ L^2(\R^n).$ It is known that when $ f \in L^1 \cap L^2(\mathbb{H}^n) $ its Fourier transform  is actually a Hilbert-Schmidt operator and one has
$$ \int_{\mathbb{H}^n} |f(z,t)|^2 dz dt = (2\pi)^{-(n+1)}\int_{-\infty}^\infty \|\widehat{f}(\lambda)\|_{HS}^2 |\lambda|^n d\lambda  $$
where $\|.\|_{HS}$ denote the Hilbert-Schmidt norm. 
The above allows us to extend  the Fourier transform as a unitary operator between $ L^2(\mathbb{H}^n) $ and the Hilbert space of Hilbert-Schmidt operator valued functions  on $ \R $ which are square integrable with respect to the Plancherel measure  $ d\mu(\lambda) = (2\pi)^{-n-1} |\lambda|^n d\lambda.$ We polarize the above identity to obtain 
$$\int_{\He^n}f(z,t)\overline{g(z,t)}dzdt=\int_{-\infty}^{\infty}tr(\widehat{f}(\lambda)\widehat{g}(\lambda)^*)~d\mu(\lambda).$$ Also for suitable function $f$ on $\He^n$ we have the following inversion formula
$$f(z,t)=\int_{-\infty}^{\infty}tr(\pi_{\lambda}(z,t)^*\widehat{f}(\lambda))d\mu(\lambda).$$
Now by definition of $\pi_{\lambda}$ and $\hat{f}(\lambda)$ it is easy to see that 
$$\widehat{f}(\lambda)=\int_{\C^n}f^{\lambda}(z)\pi_{\lambda}(z,0)dz $$ where 
$f^{\lambda}$ stands for the inverse Fourier transform of $f$ in the central variable:
$$f^{\lambda}(z):=\int_{-\infty}^{\infty}e^{i\lambda.t}f(z,t)dt.$$ This motivates the following operator. Given a function $g$ on $\C^n$, we consider the following   operator valued function  defined by
$$ W_{\lambda}(g):=\int_{\C^n}g(z)\pi_{\lambda}(z,0)dz.
$$ With these notations we note that  $\hat{f}(\lambda)=W_{\lambda}(f^{\lambda}).$  These transforms are called the Weyl transforms and for  $\lambda=1$ they are simply denoted by $ W(g) $ instead of $W_1(g).$  Moreover, the Fourier transform bahaves well with the convolution of two functions defined by $$f\ast g(x):=\int_{\He^n}f(xy^{-1})g(y)dy.$$ Infact, for any $f,g\in L^1(\mathbb{H}^n)$,   directly   from the definition it follows that $\widehat{f \ast g}(\lambda)=\hat{f}(\lambda)\hat{g}(\lambda).$ 
\subsection{Special functions and Fourier transform}
 For each $\lambda\neq0$, we consider the following  family of scaled Hermite functions indexed by $\alpha\in\mathbb{N}^n$: $$\Phi_\alpha^{\lambda}(x):=|\lambda|^{\frac{n}{4}}\Phi_\alpha(\sqrt{|\lambda| }x),~x\in\mathbb{R}^n $$ 
 where $\Phi_\alpha$ denote the $n-$dimensional Hermite functions (see \cite{TH1}). It is well-known that these  scaled functions $\Phi_\alpha^{\lambda}$ are eigenfunctions of the scaled Hermite operator $H(\lambda):=-\Delta_{\R^n}+\lambda^2|x|^2$ with eigenvalue $(2|\alpha|+n)|\lambda|$ and $\{\Phi_\alpha^{\lambda}:\alpha\in\mathbb{N}^n\}$ forms an orthonormal basis for $L^2(\R^n)$. As a consequence, 
 $$\|\widehat{f}(\lambda)\|_{HS}^2=\sum_{\alpha \in \mathbb{N}^n}\|\widehat{f}(\lambda)\Phi_\alpha^{\lambda}\|_2^2.$$ In view of this the Plancheral formula takes the following very useful form 
 $$\int_{\mathbb{H}^n} |f(z,t)|^2 dz dt =  \int_{-\infty}^\infty \sum_{\alpha \in \mathbb{N}^n}\|\widehat{f}(\lambda)\Phi_\alpha^{\lambda}\|_2^2 \  d\mu( \lambda) . $$

 Given $\sigma\in U(n)$, we define $R_{\sigma}f(z,t)=f(\sigma.z,t)$. We say that a function $f$ on $\He^n$ is radial if $f$ is invariant under the action of $U(n)$ i.e., $R_{\sigma}f=f$ for all $\sigma\in U(n).$ The Fourier transforms of such radial integrable funtions are  functions of the Hermite operator $ H(\lambda)$ .
 In fact, if  $ H(\lambda) = \sum_{k=0}^\infty (2k+n)|\lambda| P_k(\lambda)$  stands for the spectral decomposition of this operator, then for a radial intrgrable function $f$  we have
 $$ \widehat{f}(\lambda)  = \sum_{k=0}^\infty  R_k(\lambda, f) P_k(\lambda).$$ More explicitly, $P_k(\lambda)$ stands for the orthogonal projection of $L^2(\mathbb{R}^n)$ onto the $k^{th}$ eigenspace spanned by scaled Hermite functions $\Phi^{\lambda}_{\alpha}$ for $|\alpha|=k$. The coefficients $ R_k(\lambda,f) $ are explicitly given by
 \begin{equation}
R_k(\lambda,f)  =  \frac{k!(n-1)!}{(k+n-1)!} \int_{\C^n}  f^{\lambda }(z) \varphi^{n-1}_{k,\lambda}(z)~dz.
 \end{equation}
  
 In the above formula, $ \varphi_{k,\lambda}^{n-1} $ are the Laguerre functions of type $ (n-1)$:
 $$  \varphi^{n-1}_{k,\lambda}(z)  = L_k^{n-1}(\frac{1}{2}|\lambda||z|^2) e^{-\frac{1}{4}|\lambda||z|^2} $$ where $L^{n-1}_k$ denotes the Laguerre polynomial of type $(n-1)$. For the purpose of estimating the Fourier transform we need good estimates for the  Laguerre functions $\varphi_{k,\lambda}^{n-1}.$ In order to get such estimates, we use the available sharp estiamtes of standard Laguerre functions as described below in more general context.

 For any $ \delta > -1 $, let $ L_k^\delta(r) $ denote the Laguerre polynomials of type $ \delta$.  The standard Laguerre functions are defined by 
 $$
 \mathcal{L}_k^{\delta}(r)=\Big(\frac{\Gamma(k+1)\Gamma(\delta+1)}{\Gamma(k+\delta+1)}\Big)^{\frac12}L_k^{\delta}(r)e^{-\frac12r}r^{\delta/2}
 $$
 which  form an orthonormal system in $L^2((0,\infty),dr)$. In terms of $\mathcal{L}_k^{\delta}(r),$ we have
 $$
 \varphi_k^{\delta}(r)=2^{\delta}\Big(\frac{\Gamma(k+1)\Gamma(\delta+1)}{\Gamma(k+\delta+1)}\Big)^{-\frac12}r^{-\delta}\mathcal{L}_k^{\delta}\Big(\frac12r^2\Big).
 $$
 Asymptotic properties of $\mathcal{L}_k^{\delta}(r)$ are well known in the literature, see \cite[Lemma 1.5.3]{TH1}. The estimates in \cite[Lemma 1.5.3]{TH1}   are sharp, see \cite[Section 2]{M} and \cite[Section 7]{Mu}.
 For our convenience, we restate the result in terms of $ \varphi_{k,\lambda}^{n-1}(r).$
 
  \begin{lem} \label{lem:T}   Let $ \nu(k) = 2(2k+n) $ and  $C_{k,n} = \left(\frac{k!(n-1)!}{(k+n-1)!}\right)^{\frac12}.$ For  $\lambda\neq0, $ we have the estimates
    	 $$
    	 C_{k,n}\,\, |\varphi_{k,\lambda}^{n-1}(r)|\le C(r\sqrt{|\lambda|})^{-(n-1)}\begin{cases} (\frac{1}{2}\nu(k)r^2 |\lambda| )^{(n-1)/2}, &0\le r\le \frac{\sqrt{2}}{\sqrt{\nu(k)|\lambda|}}\\
    	  (\frac{1}{2}\nu(k)r^2 |\lambda|)^{-\frac14}, &\frac{\sqrt{2}}{\sqrt{\nu(k)|\lambda|}}\le r\le \frac{\sqrt{\nu(k)}}{ \sqrt{|\lambda|}}\\
    	  \nu(k)^{-\frac14}(\nu(k)^{\frac13}+|\nu(k)-\frac{1}{2} |\lambda| r^2|)^{-\frac14}, &\frac{\sqrt{\nu(k)}}{ \sqrt{|\lambda|}}\le r\le \frac{\sqrt{3\nu(k)}}{ \sqrt{|\lambda|}}\\
    	  e^{-\frac12\gamma r^2 |\lambda| }, & r\ge \frac{\sqrt{3\nu(k)}}{ \sqrt{|\lambda|}},
    	 \end{cases}
    	 $$ where $\gamma>0$ is a fixed constant and $C $  is independent of $k$ and $\lambda$. 
    \end{lem}

\subsection{The sublaplacian and Sobolev spaces on $ \He^n$}We let $ \mathfrak{h}_n $ stand for the Heisenberg Lie algebra consisting of left invariant vector fields on $ \mathbb{H}^n .$  A  basis for $ \mathfrak{h}_n $ is provided by the $ 2n+1 $ vector fields
$$ X_j = \frac{\partial}{\partial{x_j}}+\frac{1}{2} y_j \frac{\partial}{\partial t}, \,\,Y_j = \frac{\partial}{\partial{y_j}}-\frac{1}{2} x_j \frac{\partial}{\partial t}, \,\, j = 1,2,..., n $$
and $ T = \frac{\partial}{\partial t}.$  These correspond to certain one parameter subgroups of $ \mathbb{H}^n.$ The sublaplacian on $\He^n$ is defined by $$\mathcal{L}:=-\sum_{j=1}^{\infty}(X_j^2+Y_j^2) $$ which can be explicitly calculated as
$$\mathcal{L}=-\Delta_{\C^n}-\frac{1}{4}|z|^2\frac{\partial^2}{\partial t^2}+N\frac{\partial}{\partial t}$$ where  $\Delta_{\C^n}$ stands for the Laplacian on $\C^n$ and $N$ is the rotation operator defined by 
$$N=\sum_{j=1}^{n}\left(x_j\frac{\partial}{\partial y_j}-y_j\frac{\partial}{\partial x_j}\right).$$ This is a sub-elliptic operator and homogeneous of degree $2$ with respect to the non-isotropic dilations given by $\delta_r(z,t)=(rz,r^2t).$ The sublaplacian is also invariant undrer  rotation i.e., $R_{\sigma}\circ \mathcal{L}=\mathcal{L}\circ R_{\sigma},~\sigma\in U(n).$
It is convenient for our purpose to represent the sublaplacian in terms of  another set of vector fields defined as follows:
$$Z_j:=\frac{1}{2}(X_j-iY_j)=\frac{\partial}{\partial z_j}+\frac{i}{4}\bar{z}_j\frac{\partial }{\partial t},~\bar{Z}_j:=\frac{1}{2}(X_j+iY_j)=\frac{\partial}{\partial \bar{z}_j}+\frac{i}{4}z_j\frac{\partial }{\partial t}$$ where $\frac{\partial}{\partial z_j}=\frac{1}{2}\left(\frac{\partial}{\partial x_j}-i\frac{\partial}{\partial y_j}\right)$ and $\frac{\partial}{\partial \bar{z}_j}=\frac{1}{2}\left(\frac{\partial}{\partial x_j}+i\frac{\partial}{\partial y_j}\right).$ Now an easy calculation yields 
$$\mathcal{L}=-\frac12\sum_{j=1}^n\left(Z_j\bar{Z}_j+\bar{Z}_jZ_j\right).$$
The action of Fourier transform on  $Z_jf$ , $\bar{Z}_jf$ and $Tf$ are well-known and  are given by 
\begin{equation}
\label{Frel}
\widehat{Z_jf}(\lambda)=i\widehat{f}(\lambda)A_j(\lambda)~,~\widehat{\bar{Z}_jf}(\lambda)=i\widehat{f}(\lambda)A_j(\lambda)^*~\text{and}~\widehat{Tf}(\lambda)=-i\lambda\widehat{f}(\lambda)
\end{equation}
   where $A_j(\lambda)$ and $A_j^*(\lambda)$ are the annihilation and creation operators given by $$A_j(\lambda)=\Big(-\frac{\partial}{\partial\xi_j}+i\lambda\xi_j\Big),~~~~A_j^*(\lambda)=\Big(\frac{\partial}{\partial\xi_j}+i\lambda\xi_j\Big).$$ 
These along with  the above representation of the sublaplacian yield 
the relation $\widehat{\mathcal{L}f}(\lambda)=\widehat{f}(\lambda)H(\lambda).$

 We can define the spaces $ W^{s,2}(\He^n) $ for any $ s \in \R $ as the completion of $ C_c^\infty(\He^n) $ under the norm $ \| f\|_{(s)} = \| (I+\mathcal{L})^{s/2}f \|_2 $ where the fractional powers $ (I+\mathcal{L})^{s/2} $ are defined using spectral theorem. 
To study these spaces, it is better to work with the following expression of the norm $ \| f\|_{(s)} $ for $ f \in C_c^\infty(\He^n).$ In view of Plancherel theorem for the Fourier transform
$$ \|f \|_{(s)}^2  = (2\pi)^{-n-1} \int_{-\infty}^\infty  \| \widehat{f}(\lambda) (1+H(\lambda))^{s/2}\|_{HS}^2 |\lambda|^n d\lambda $$
which is valid for any $ s \in \R.$ Here we have made use of the fact that $ \widehat{\mathcal{L}f}(\lambda) = \widehat{f}(\lambda)H(\lambda).$  Computing the Hilbert-Schmidt norm in terms of the Hermite basis, we have the more explicit expression:
$$ \|f \|_{(s)}^2  = (2\pi)^{-n-1} \int_{-\infty}^\infty  \sum_{\alpha \in \mathbb{N}^n}\sum_{\beta \in \mathbb{N}^n} (1+(2|\alpha|+n)|\lambda|)^s | \langle \widehat{f}(\lambda)\Phi_\alpha^\lambda, \Phi_\beta^\lambda \rangle |^2 |\lambda|^n d\lambda.$$
Consider $ \widehat{\He}^n= \R^\ast \times \mathbb{N}^n \times \mathbb{N}^n $  equipped with the measure $ \mu \times \nu $ where $ \nu $ is the counting measure on $ \mathbb{N}^n \times \mathbb{N}^n.$ The above shows that, for  $ f \in C_c^\infty(\He^n) $ the function  $ m(\lambda,\alpha,\beta) = \langle \widehat{f}(\lambda)\Phi_\alpha^\lambda, \Phi_\beta^\lambda \rangle.$ belongs to the weighted space 
$$ W^{s,2}(\widehat{\He}^n) = L^2(\widehat{\He}^n, w_s\, d(\mu \times \nu)) $$ where  $ w_s(\lambda,\alpha) = (1+(2|\alpha|+n)|\lambda|)^{s}.$  As these weighted $ L^2 $ spaces are complete, we can identify $ W^{s,2}(\He^n) $ with $ W^{s,2}(\widehat{\He}^n).$ It is then clear that for any $ s>0$ we have 
$$  W^{s,2}(\widehat{\He}^n)  \subset W^{0,2}(\widehat{\He}^n) \subset W^{-s,2}(\widehat{\He}^n) $$  and the same inclusion holds for $ W^{s,2}(\He^n).$  It is clear that  any $ m \in W^{s,2}(\widehat{\He}^n) $ can be written as $ m(\lambda, \alpha,\beta) = (1+(2|\alpha|+n)|\lambda|)^{-s/2} m_0(\lambda,\alpha,\beta) $ where $ m_0 \in W^{0,2}(\widehat{\He}^n)= L^2(\widehat{\He}^n)  $ for any $ s \in \R.$ Consequently, any $ f \in W^{s,2}(\He^n) $ can be written as  $ f = (I+\mathcal{L})^{-s/2}f_0 $, where $ f_0 \in L^2(\He^n) $ is the function which corresponds to $ m_0$  which is given explicitly by
$$ f_0(z,t)  = \int_{\widehat{\He}^n} m_0(\lambda,\alpha,\beta) e_{\alpha,\beta}^{-\lambda}(z,t) d\nu(\alpha,\beta) d\mu(\lambda).$$

Thus we see that $ f \in W^{s,2}(\He^n) $ if and only if there is an $ f_0 \in L^2(\He^n) $ such that  $ f = (I+\mathcal{L})^{-s/2}f_0 .$ The inner product on $ W^{s,2}(\He^n) $ is given by 
$$ \langle f, g \rangle_s =  \langle (I+\mathcal{L})^{s/2}f,(I+\mathcal{L})^{s/2}g \rangle = \langle f_0,g_0\rangle $$ where $ \langle f, g \rangle $ is the inner product in $ L^2(\He^n).$  This has the following interesting consequence. Given  $ f \in W^{s,2}(\He^n) $ and $ g \in W^{-s,2}(\He^n), $  let $ f_0, g_0 \in L^2(\He^n) $ be such that $ f = (I+\mathcal{L})^{-s/2}f_0 $ and $ g = (I+\mathcal{L})^{s/2}g_0.$ The duality bracket $ (f,g) $ defined by
$$ (f,g) = \langle (I+\mathcal{L})^{-s/2}f_0, (I+\mathcal{L})^{s/2}g_0 \rangle  = \langle f_0, g_0\rangle $$  allows us to identify the dual of $ W^{s,2}(\He^n) $ with $ W^{-s,2}(\He^n).$ This is also clear from the identification of $ W^{s,2}(\He^n) $ with $ W^{s,2}(\widehat{\He}^n).$ Thus for every $ g \in W^{-s,2}(\He^n) $ there is a linear functional $ \Lambda_g : W^{s,2}(\He^n) \rightarrow \C  $ given by $ \Lambda_g(f) = \langle f_0,g_0\rangle. $ 

The following observation is also very useful in applications. For  $ s > 0 $ every member  $ f \in W^{s,2}(\He^n)$ defines a distribution on $ \He^n.$  The same is true for every $ g \in W^{-s,2}(\He^n)$ as well. To see this, consider the map taking $ f \in C_c^\infty(\He^n) $ into the duality bracket $ (f,g)$ which satisfies
$$  |(f,g)| \leq \| f\|_{(s)} \|g\|_{(-s)} \leq \|g\|_{(-s)} \| (I+\mathcal{L})^m f\|_2 $$
where $ m > s/2 $ is an integer. From the  above it is clear that $ \Lambda_g(f) = (f,g) $ is a distribution. If $ g \in W^{-s,2}(\He^n) $ is such a distribution, it is possible to define its Fourier transform as an unbounded operator valued function on $ \R^\ast.$ Indeed, let $ g_0 \in L^2(\He^n) $ be such that $ g = (I+\mathcal{L})^{s/2}g_0 $ then we define $ \widehat{g}(\lambda) =  \widehat{g_0}(\lambda) 
(1+H(\lambda))^{s/2}$ which  is a densely defined operator whose action on $ \Phi_\alpha^\lambda $ is given by
$$ \widehat{g}(\lambda) \Phi_\alpha^\lambda = (1+(2|\alpha|+n)|\lambda|)^{s/2} \widehat{g_0}(\lambda) \Phi_\alpha^\lambda.$$
Thus we see that when $ g \in W^{-s,2}(\He^n) $ we have 
\begin{equation}\label{norm-remark}  \int_{-\infty}^\infty  \sum_{\alpha \in \mathbb{N}^n}\sum_{\beta \in \mathbb{N}^n} (1+(2|\alpha|+n)|\lambda|)^{-s} | \langle \widehat{g}(\lambda)\Phi_\alpha^\lambda, \Phi_\beta^\lambda \rangle |^2  d\mu(\lambda) = \int_{\He^n} |g_0(z,t)|^2 dz dt < \infty. \end{equation}

\begin{rem}\label{rem} When $ g \in W^{-s,2}(\He^n) $ is a compactly supported distribution, then we already have a definition of $ \widehat{g}(\lambda) $ given by
$  \langle \widehat{g}(\lambda) \Phi_\alpha^\lambda, \Phi_\beta^\lambda \rangle = ( g, e_{\alpha,\beta}^\lambda),$  the action of $ g $ on the smooth function $ e_{\alpha,\beta}^\lambda (z,t).$
The two definitions agree as $ e_{\alpha,\beta}^\lambda $ are eigenfunctions of $ \mathcal{L} $ with eigenvalues $ (2|\alpha|+n)|\lambda|.$
\end{rem}

\section{Chernoff's theorem on the Heisenberg group}
In this section we prove Theorem \ref{CH} for the sublaplacian on the Heisenberg group. For the proof we need to recall  some properties of the so called  Stieltjes vectors for the sublaplacian.

Let $X$ be a Banach space and $A,$  a linear operator on $X$ with domain $D(A)\subset X.$ A vector $x\in X$ is called a $C^{\infty}$- vector or smooth vector for $ A$ if $x\in \cap_{n=1}^{\infty}D(A^n).$ A $C^{\infty}$- vector $x$ is said to be a Stieltjes vector for $A$ if $\sum_{n=1}^{\infty}\|A^nx\|^{-\frac{1}{2n}}=\infty.$ These vectors were first introduced  by Nussbaum \cite{ANs} and independently by Masson and Mc Clary \cite{MM}. We denote the set of all Stieltjes vector for $A$ by $D_{St}(A).$ The following  theorem summarises  the interconnection  between the theory of Stieltjes vectors and the essential self adjointness of certain class of operators. 
\begin{thm}\label{S}
	Let $A$ be a semibounded symmetric operator on a Hilbert space $H$. Assume that the set $D_{St}(A)$ has a dense span. Then $A$ is essentially self adjoint.
\end{thm} 
A very nice simplified proof this theorem can be found in  Simon \cite{BS}. In 1975, P.R.Chernoff used this result to prove an $L^2$-version of the classical Denjoy-Carleman theorem regarding quasi-analytic functions on $\mathbb{R}^n.$

The above theorem talks about essential self adjointness of operators. Let us quickly recall some relevant definitions from operator theory. By an operator $A$ on a Hilbert space $\mathcal{H}$ we mean a linear mapping whose domain $D(A)$ is a subspace of $\mathcal{H}$ and whose range $Ran(A)\subset \mathcal{H}$. We say that an operator $S$ is an extension of $A$ if $D(A)\subset D(S)$ and $Sx=Ax $ for all $  x\in D(A)$. An operator $A$ is called closed if the graph of $A$ defined by $\mathcal{G}=\{(x,Ax):x\in D(A)\}$ is a closed subset of $ \mathcal{H} \times \mathcal{H}.$  We say that an operator $A$ is closable if it has a closed extension. Every closable operator has a smallest closed extension, called its closure, which we denote by $\bar{A}.$ An operator $A$ is said to be densely defined if $D(A)$ is dense in $\mathcal{H}$ and it is called symmetric if $\langle Ax,y\rangle = \langle x,Ay \rangle$ for all $x,y\in D(A)$. A densely defined symmetric operator $A$ is called essentially self adjoint if its closure $\bar{A}$ is self adjoint. It is easy to see that an operator $A$ is essentially self adjoint if and only if $A$ has unique self adjoint extension. The following is a very important characterization of essentially self adjoint operators. 
\begin{thm}(\cite{RB})\label{EAO}
	Let $A$ be a positive, densely defined symmetric operator. The followings are equivalent: (i) $A$ is essentially self adjoint (ii)  $Ker(A^*+I)=\{0\}$ and (iii)  $Ran(A+I)$ is dense in $\mathcal{H}.$
\end{thm}

We apply the above theorem to study the essential self adjointness of $ \mathcal{L} $ considered on a domain inside the Sobolev space $ W^{s,2}(\He^n), s >0. $  Let  $ A $ stand for the sublaplacian $ \mathcal{L} $ restricted to the domain $ D(A) $ consisting of all smooth functions $ f $ such that for all $ \alpha, \beta \in \mathbb{N}^n, j \in \mathbb{N} $ the derivatives $ X^\alpha Y^\beta T^j f  $ are in $ L^2(\He^n) $ and vanish at the origin. Since $ X_j,Y_j $ agree with $ \frac{\partial}{\partial x_j}, \frac{\partial}{\partial y_j}$ at the origin, we can also define $ D(A) $ in terms of ordinary derivatives $ \partial_x^\alpha \partial_y^\beta \partial_t^j.$

\begin{prop}\label{essen}  Let $ A $ and $ D(A) $ be defined as above where $ (n-1) < s \leq (n+1).$  Then $ A $ is not essentially self adjoint.
\end{prop}
\begin{proof} In view of Theorem \ref{EAO} it is enough to show that for $ s $ as in the statement of the proposition,  $ D(A) $ is dense in $ W^{s,2}(\He^n) $ but $ (I+A) D(A) $ is not. These are proved in the following lemmas.
\end{proof}

\begin{lem}\label{Density}  $ D(A) $ is dense in $ W^{s,2}(\He^n) $ for any $ 0 \leq s \leq (n+1). $ 
\end{lem}
\begin{proof} If we let  $ \Omega= \mathbb{H}^n\setminus\{0\}$ so that $ C_c^\infty(\Omega)  \subset D(A) ,$  it is enough to show that the smaller set is  dense in $ W^{s,2}(\mathbb{H}^n).$  This will follow if  we can show that the only linear functional that annihilates  $ C_c^{\infty}(\Omega)$ is the zero functional (see chapter 3 of \cite{WR}). Let $\Lambda \in  (W^{s,2}(\He^n))' $, the dual of $W^{s,2}(\He^n)$,  be such that $\Lambda ( C_c^{\infty}(\Omega))=0.$ Then there exists $g\in W^{-s,2}(\He^n) $ such that $ \Lambda = \Lambda_g$ and  hence  $  \Lambda_{g}(\phi)= 0 $ for any  $ \phi \in  C_c^{\infty}(\Omega)$   
Notice that for $\phi \in C_c^{\infty}(\mathbb{H}^n)$ the linear map $\phi \rightarrowtail \Lambda_{g}(\phi)$ defines a distribution. Indeed, the estimate
$$ |\Lambda_g(\phi)| \leq \|g\|_{(-s)} \|\phi\|_{(s)} \leq \|g\|_{(-s)} \| (I+\mathcal{L})^m\phi \|_2 $$
for any integer $ m > s/2 $ shows that it is indeed a distribution. As it vanishes on $ \Omega $ it is supported at the origin. The structure theory of such distributions allow us to conclude that 
$ \Lambda_g $ is a finite linear combination of derivatives of Dirac $ \delta $ at the origin, $\Lambda_{g}=\sum_{|a|\leq N}c_{a}\partial^{a}\delta,$ see e.g Chapter 6 of \cite{WR}.

Since $ X^a \delta = \partial_x^a \delta ,$ and $ Y^b \delta = \partial_y^b \delta $ in the above representation we can also use $ X^a Y^bT^j.$ It is even more convenient to write them in terms of the complex vector fields defined by $ Z_j = \frac{1}{2} (X_j-iY_j), \overline{Z}_j = \frac{1}{2}(X_j+iY_j).$ Thus we have 
$ g = \sum_{|a|+|b|+2j \leq N} c_{a,b,j} Z^a \overline{Z}^b T^j \delta.$ If $ g_0 \in L^2(\He^n) $ is such that $ (I+\mathcal{L})^{-s/2}g = g_0 $ then by (\ref{norm-remark}) we have
$$  \int_{-\infty}^\infty  \sum_{\alpha \in \mathbb{N}^n}\sum_{\beta \in \mathbb{N}^n} (1+(2|\alpha|+n)|\lambda|)^{-s} | \langle \widehat{g}(\lambda)\Phi_\alpha^\lambda, \Phi_\beta^\lambda \rangle |^2  d\mu(\lambda) < \infty.$$
Since $ g $ is compactly supported we can calculate the Fourier transform of $ g $ as in  Remark \ref{rem}. In view of the relations \ref{Frel} we have
$$ \langle \widehat{g}(\lambda) \Phi_\alpha^\lambda,\Phi_\beta^\lambda \rangle = \sum_{|a|+|b|+2j \leq N} c_{a,b,j} \lambda^j   \langle A(\lambda)^a (A(\lambda)^\ast)^b \Phi_\alpha^\lambda, \Phi_\beta^\lambda \rangle$$
By defining $ m(\lambda,\alpha, \beta) $ to be the expression on the right hand side of the above equation we  see that
\begin{equation}\label{express}  \int_{-\infty}^\infty  \sum_{\alpha \in \mathbb{N}^n}\sum_{\beta \in \mathbb{N}^n} (1+(2|\alpha|+n)|\lambda|)^{-s}  | m(\lambda,\alpha,\beta)|^2 |\lambda|^n d\lambda < \infty.
\end{equation}
The action of $ A(\lambda)^a $ and $ (A(\lambda)^\ast)^b$ on $ \Phi_\alpha^\lambda $ are explicitly known, see (\cite{TH3}). It is therefore easy to see that 
$$ m((2|\alpha|+n)^{-1}\lambda, \alpha,\beta)  =  \sum_{|a|+|b|+2j \leq N} C_{a,b,j}( \alpha,\beta)  \lambda^{j+(|a|+|b|)/2}  $$
where the coefficients $ C_{a,b,j}(\alpha,\beta) $ are uniformly bounded in both  variables. We also remark that  for a given $ \alpha $ the function $ C_{a,b,j}(\alpha,\beta) $ is non zero only for a single value of $ \beta.$
By making a change of variables in (\ref{express}) we see that
$$   \sum_{\alpha \in \mathbb{N}^n}\sum_{\beta \in \mathbb{N}^n} (2|\alpha|+n)^{-n-1}  \int_{-\infty}^\infty  \big( \sum_{|a|+|b|+2j \leq N} C_{a,b,j}( \alpha,\beta)  \lambda^{j+(|a|+|b|)/2} \big)^2 \frac{|\lambda|^n}{(1+|\lambda|)^s}  d\lambda < \infty.$$
As we are assuming that $  0 \leq s \leq (n+1)$ the above integral cannot be finite unless all the coefficients $ c_{a,b,j} =0.$ Hence $ g =0 $ proving the density of $ D(A).$
\end{proof}

\begin{lem} For any $ s > (n-1),  (I+A)D(A) $ is not dense in $ W^{s,2}(\He^n).$
\end{lem}
\begin{proof} For any $ f \in D(A)$ the inversion formula for the Fourier transform on $ \He^n $ shows that 
$$ \int_{-\infty}^\infty  tr( \widehat{f}(\lambda)) d\mu(\lambda) = f(0) = 0.$$
Let $ g $ be the functions defined by $ \widehat{g}(\lambda) =  (1+H(\lambda))^{-s-1}$ we can rewrite the above as 
$$ \langle (I+\mathcal{L})f, g\rangle_s =  \int_{-\infty}^\infty  tr( \widehat{f}(\lambda)) d\mu(\lambda) = 0.$$
So all we need to do is to check $ g \in W^{s,2}(\He^n),$ or equivalently
$$  \int_{-\infty}^\infty \big(  \sum_{k=0}^\infty (1+(2k+n)|\lambda|)^{-s-2} \| P_k(\lambda)\|_{HS}^2 \big)  |\lambda|^n d\lambda< \infty.$$
It is known that  $ \|P_k(\lambda)\|_{HS}^2 = \frac{(k+n-1)!}{k!(n-1)!} \leq C (2k+n)^{n-1} $ and so by making a change of variables the above integral is bounded by
$$  \sum_{k=0}^\infty  (2k+n)^{-2} \int_{-\infty}^\infty (1+|\lambda|)^{-s-2} |\lambda|^n d\lambda.$$
As we assume that $ s > (n-1) $ the integral is finite which proves that $ g \in W^{s,2}(\He^n).$ Hence the lemma.
\end{proof}  

We now proceed to investigate some properties of the set $ D_{St}(A) $ of Stieltjes vectors for the operator $ A.$  The following lemma about series of real numbers will be helpful in proving some properties of Stieltjes vectors for the sublaplacian (see lemma 3.2 of \cite{CH1}).
\begin{lem}\label{series}
If $\{M_n\}_n$ is sequence of non-negetive real numbers such that $\sum_{n=1}^{\infty}M_n^{-\frac{1}{n}}=\infty$ and $0\leq K_n\leq aM_n+b^n$, then $\sum_{n=1}^{\infty}K_n^{-\frac{1}{n}}=\infty.$
\end{lem}

 For $r>0 $  the non-isotropic dilation of $f$ is defined by by $\delta_rf(z,t)=f(rz,r^2t)$ and for $\sigma\in U(n)$ we define the rotation $R_{\sigma}f(z,t)=f(\sigma z,t)$ for all $(z,t)\in \mathbb{H}^n.$ 

\begin{lem}\label{stiel} Suppose $ f \in D(A) $ satisfies the condition $ \sum_{m=0}^\infty \| \mathcal{L}^m f\|_2^{-\frac{1}{2m} } = \infty .$ Then $ f \in D_{St}(A) .$   Moreover, $ \delta_r f, R_\sigma f $ are also Stieltjes vectors for $ A.$
\end{lem}
\begin{proof} We first recall that $ \mathcal{L} \circ \delta_r = r^2  \delta_r \circ \mathcal{L} $ and $ \mathcal{L}\circ R_\sigma = R_\sigma \circ \mathcal{L},$ see e.g \cite{TH2}. Therefore, it follows that if a function satisfies $ \sum_{m=0}^\infty \| \mathcal{L}^m f\|_2^{-\frac{1}{2m} } = \infty $ then the same is true of $ \delta_r f $ and $ R_\sigma f.$  So we only need to prove our claim for $ f;$ i.e.,  when $ f $ satisfies the above condition then we also have $ \sum_{m=0}^\infty \| \mathcal{L}^m f\|_{(s)}^{-\frac{1}{2m} } = \infty .$ To see this, we use
$$ \langle \mathcal{L}^m f,\mathcal{L}^m f \rangle_{(s)} = \langle \mathcal{L}^{2m}f,(1+\mathcal{L})^{s} f\rangle \leq \|(I+\mathcal{L})^s f\|_2 \| \mathcal{L}^{2m} f \|_2.$$
Thus we have  $ \| \mathcal{L}^m f\|_{(s)}^{-\frac{1}{2m}} \geq C^{-\frac{1}{4m}} \| \mathcal{L}^{2m}f \|_2^{-\frac{1}{4m}}$ where $ C = \|(I+\mathcal{L})^s f\|_2.$ In view of Lemma \ref{series} it is enough to prove the divergence of $ \sum_{m=0}^\infty \| \mathcal{L}^{2m} f\|_{2}^{-\frac{1}{4m} } .$ Without loss of generality we can assume that $ \|f \|_2 =1 .$ But then  $ \| \mathcal{L}^mf\|^{-\frac{1}{2m} }$ is a decreasing function of $m,$  see Lemma 2.1 in \cite{CH1}. Consequently, the required divergence follows from the assumption on $ f.$
\end{proof}

Before stating the next lemma, let us recall some properties of the matrix coefficients $ e_{\alpha,\beta}^\lambda(z,t) =\langle \pi_\lambda(z,t) \Phi_\alpha^\lambda, \Phi_\beta^\lambda) $ of the Schr\"odinger representations. These are eigenfunctions of the sublaplacian with eigenvalues $ (2|\alpha|+n)|\lambda|.$ Moreover, they satisfy
\begin{equation}\label{eigen-eq} Z_je^{\lambda}_{\alpha,\beta}=i(2\alpha_j+2)^{\frac{1}{2}}|\lambda|^{\frac{1}{2}}e^{\lambda}_{\alpha+e_j,\beta},\,\,\,\,\overline{Z}_je^{\lambda}_{\alpha,\beta}=i(2\alpha_j)^{\frac{1}{2}}|\lambda|^{\frac{1}{2}}e^{\lambda}_{\alpha-e_j,\beta}\end{equation}
where $ e_j $  are the coordinate vectors in $ \C^n.$  We also recall that the sublaplacian is expressed as  
$\mathcal{L}=-\frac{1}{2}\sum_{j=1}^{n}(\overline{Z}_jZ_j+Z_j\overline{Z}_j)$ in terms of $ Z_j $ and $ \overline{Z_j}.$

\begin{lem}\label{matrix} If $ f $ satisfies the hypothesis in Lemma \ref{stiel} , then $ e_{\alpha,\beta}^\lambda f  \in D_{St}(A),$ for any $ \alpha, \beta \in \mathbb{N}^n $ and $ \lambda \in \R^\ast.$
\end{lem}
\begin{proof} As noted  in the previous lemma,  it suffices to show that   $\sum_{m=1}^{\infty}\|\mathcal{L}^m(e^{\lambda}_{\alpha,\beta}f)\|_2^{-\frac{1}{2m}}=\infty$.   
Since $\mathcal{L}=-\frac{1}{2}\sum_{j=1}^{n}(\overline{Z}_jZ_j+Z_j\overline{Z}_j)$ in terms of $ Z_j $, a simple calculation shows that 
\begin{equation}\label{prod}  \mathcal{L}(fg) = (\mathcal{L}f) g + f (\mathcal{L}g) -\frac{1}{2} \sum_{j=1}^n \left( Z_jf \bar{Z}_j g + \bar{Z}_j f Z_jg \right).\end{equation}
By taking $ g = e_{\alpha,\beta}^\lambda $ and making use of (\ref{eigen-eq}) along with the estimate $ \|e_{\alpha,\beta}^\lambda\|_{\infty} \leq 1$ we infer that
$ \| \mathcal{L}( f e_{\alpha,\beta}^\lambda) \|_2  $ is bounded by
$$  \| \mathcal{L}f\|_2 +(2|\alpha|+n)|\lambda| \|f\|_2 + \frac{1}{\sqrt{2}} \sum_{j=1}^n \big(  \sqrt{(\alpha_j+1)|\lambda|} \|Z_jf\|_2+ \sqrt{\alpha_j |\lambda|}\|\bar{Z}_jf\|_2 \big).  $$
As the operators $ Z_j \mathcal{L}^{-1/2} $ and $ \bar{Z}_j\mathcal{L}^{-1/2} $ are bounded on $ L^2(\He^n) $ with norms at most $ \sqrt{2} $, we see that the third term above can be estimated as
$$  \sum_{j=1}^n \big(  \sqrt{(\alpha_j+1)|\lambda|} + \sqrt{\alpha_j |\lambda|} \big) \| \mathcal{L}^{1/2}f \|_2 \leq 2 \sum_{j=1}^n \big((2\alpha_j+1)|\lambda| \big) \| \mathcal{L}^{1/2}f\|_2.$$
Finally using the fact that $ \| \mathcal{L}^{1/2}(1+\mathcal{L})^{-1} f \|_2 \leq \|f\|_2 $ we get the estimate
$$ \| \mathcal{L}(f e_{\alpha,\beta}^\lambda) \|_2 \leq  (2|\alpha|+n)|\lambda| ( 2 \|\mathcal{L}f\|_2 + 3 \|f\|_2) + \| \mathcal{L}f\|_2 .$$
By defining $ a_\lambda(\alpha) = (2|\alpha|+n)|\lambda|), b_\lambda(\alpha) = (2|\alpha|+n+1)|\lambda|)$ and $  c_\lambda(\alpha) = 3b_\lambda(\alpha)+1,$    we rewrite the above as
$$ \| \mathcal{L}(f e_{\alpha,\beta}^\lambda) \|_2 \leq  c_\lambda(\alpha) \big( \|\mathcal{L}f\|_2 +  \|f\|_2 \big).$$
        
In order to prove the lemma it is enough to show for any non-negative integer $ m $ the following estimate holds:
\begin{equation}\label{claim} \| \mathcal{L}^m(f e_{\alpha,\beta}^\lambda) \|_2 \leq  2^{m-1}c_\lambda(\alpha)^m \big( \|\mathcal{L}^mf\|_2 + \|f\|_2).\end{equation}
We prove this by induction. Assuming the result for any $ m ,$ we write $ \mathcal{L}^{m+1}(fg) = \mathcal{L}^m \mathcal{L}(fg) $ and make use of (\ref{prod}) with $ g = e_{\alpha,\beta}^\lambda.$
The first two terms  $ \mathcal{L}^m(\mathcal{L}f g) $ and $\mathcal{L}^m(f \mathcal{L}g) $ together give the estimate
$$ 2^{m-1} c_{\lambda}(\alpha)^m  \big(\|\mathcal{L}^{m+1}f\|_2 +  \| \mathcal{L}f\|_2\big) + a_\lambda(\alpha)  2^{m-1}c_{\lambda}(\alpha)^m ( \|\mathcal{L}^{m}f\|_2 +  \| f\|_2\big) . $$
The boundedness of $ \mathcal{L}(1+\mathcal{L}^{m+1})^{-1} $ and $ \mathcal{L}^m(1+\mathcal{L}^{m+1})^{-1} $ allows us to bound the above by
\begin{equation}\label{one} 2^m c_\lambda(\alpha)^m (1+b_\lambda(\alpha)) \big( \| \mathcal{L}^{m+1}f\|_2 +\|f\|_2 \big).\end{equation}
We now turn our attention to the estimation of the term
$$ \frac{1}{\sqrt{2}} \sum_{j=1}^n  \big(  \sqrt{(\alpha_j+1)|\lambda|} \mathcal{L}^m(Z_jf  e_{\alpha+e_j,\beta}^\lambda)+ \sqrt{\alpha_j |\lambda|} \mathcal{L}^m(e_{\alpha-e_j,\beta}^\lambda \bar{Z}_jf)  \big).  $$
By using the induction hypothesis along with the fact  that the operators $ \mathcal{L}^m Z_j (1+\mathcal{L}^{m+1})^{-1} $ and   $ \mathcal{L}^m \bar{Z}_j(1+\mathcal{L}^{m+1})^{-1}$ are bounded with norm  at most $ \sqrt{2} $ the $ L^2 $ norm of the above is bounded by
$$  2^{m-1} \sum_{j=1}^n  \big(  c_\lambda(\alpha+e_j)^m \sqrt{(\alpha_j+1)|\lambda|} + c_\lambda(\alpha-e_j)^m \sqrt{\alpha_j |\lambda|} \big)  \big(  \| \mathcal{L}^{m+1}f \|_2+  \| f\|_2 \big).$$
Since $ b_\lambda(\alpha+e_j) \leq 2 b_\lambda(\alpha),$ we have $ c_\lambda(\alpha+e_j) \leq 2 c_{\lambda}(\alpha),$ and so  the above sum is bounded by
\begin{equation}\label{two}    2 a_\lambda(\alpha)  2^{m}c_\lambda(\alpha)^m \big(  \| \mathcal{L}^{m+1}f \|_2+  \| f\|_2 \big).
\end{equation}
Combining (\ref{one}) and (\ref{two}), using $ a_\lambda(\alpha) \leq b_\lambda(\alpha) $ and recalling the definition of $ c_\lambda$ we obtain (\ref{claim}), proving the lemma.
\end{proof}

\begin{prop}\label{non-essen} Let $ A $ be as in Proposition \ref{essen} where we have assumed that $ (n-1) < s \leq (n+1).$
 Assume that  $ D_{St}(A) $ contains a subset $ V $ which has the following properties: (i)  every element of $ V $ satisfies the hypothesis of Lemma \ref{stiel} (ii) for every $ (z,t) \in \He^n $ there exists $ f \in V $ such that $ f(z,t) \neq 0.$ Then the linear span of $ D_{St}(A) $ is dense in $ W^{s,2}(\He^n) .$ \end{prop}
\begin{proof}  We first observe that if $ f \in V$ then $\delta_r f$ and $ R_\sigma f $ are also in $ V.$  We will show that the closed linear span of  $ D_{St}(A) ~\text{equals}~ W^{s,2}(\He^n).$  To prove this, let us take $ g \in W^{s,2}(\He^n) $ which is orthogonal to $ D_{St}(A).$ By Lemma \ref{matrix} we know that $ f  e_{\alpha,\beta}^\lambda  \in D_{St}(A) $ for all $ \alpha, \beta \in \mathbb{N}^n $ and $ \lambda \in \R^\ast.$ Thus,
$$  \langle (I+\mathcal{L})^sg,e^{\lambda}_{\alpha\beta}f\rangle_{L^2}= \langle g,  e^{\lambda}_{\alpha\beta}f\rangle_{(s)}=  0.$$
 By defining $ p(z,t)=f(z,t)(I+\mathcal{L})^sg(z,t) $,  the above translates into 
$$  \langle \widehat{p}(\lambda) \Phi_\alpha^\lambda, \Phi_\beta^\lambda \rangle =  \int_{\mathbb{H}^n}p(z,t)(\pi_{\lambda}(z,t)\Phi^{\lambda}_{\alpha},\Phi^{\lambda}_{\beta})dzdt = 0.$$
By the inversion formula for the Fourier transform on $ \He^n $ we conclude that $ p = 0 $ which means $ (1+\mathcal{L})^{s} g $ vanishes on the support of $ f$. Under the assumption on $ V$ it follows that $(1+\mathcal{L})^s g(z,t)=0 $ for every $ (z,t) \in \He^n $ from which we can conclude that $ g =0 $ as the operator $ (1+\mathcal{L})^s $ is invertible. This proves the density. 
\end{proof}

Finally, we are in a position to prove  the analogue of  Chernoff's theorem for the sublaplacian on the Heisenberg group.

{\bf{ Proof of Theorem \ref{CH}}} Consider the operator $ A $ defined in Proposition \ref{essen}. We have already shown that it is not essentially self adjoint. Suppose there exists a nontrivial $ f $ satisfying the hypothesis of Theorem \ref{CH}. Then by Lemma \ref{stiel} we know that $ f $ along with $ \delta_r f $ and $ R_\sigma f $ belong to $ D_{St}(A).$  We take $ V = \{ R_\sigma(\delta_r f): \sigma \in U(n), r> 0 \}.$  It is clear that $ V $ is invariant under $ R_\sigma $ and $ \delta_r.$ If we can show that $ V $ also satisfies the condition $ (ii)$ in  Proposition \ref{non-essen} we know that the linear span of $ D_{St} $ is dense in $ W^{s,2}(\He^n).$ By Theorem \ref{S} this allows us to conclude that $ A $ is essentially self-adjoint. As this is not the case, $ f $ has to be trivial which proves the theorem.

Thus it remains to prove the claim. Let $ (w,s) =( \rho^\prime \sqrt{\sin \theta^\prime} \, \omega^\prime, (\rho^\prime)^2\, \cos \theta^\prime) $ be such that $ f(w,s) \neq 0.$ For any $ (z,t) \in \He^n $ write $ (z,t) = (\rho \sqrt{\sin \theta} \, \omega, \rho^2 \cos \theta)$  and choose $ \sigma \in U(n) $ such that $ \sigma \cdot \omega = \omega^\prime.$ Then it is clear that 
$$ R_\sigma(\delta_{\rho^\prime/\rho}f)(z,t) = f(\rho^\prime \, \sqrt{\sin \theta} \, \omega^\prime, (\rho^\prime)^2\, \cos \theta)) = f(\rho^\prime \sqrt{\sin \theta^\prime} \, \omega^\prime, (\rho^\prime)^2\, \cos \theta^\prime) .$$ The extra assumption on $ f $ means that $ R_\sigma(\delta_{\rho^\prime/\rho}f)(z,t) = f(w,s) \neq 0.$ Hence the claim.

\section{Ingham's theorem on the Heisenberg group}   In this section we prove Theorem \ref{ingh-hei} using Chernoff's theorem for the sublaplacian. We first show the existence of a compactly supported function $ f $ on $ \He^n $ whose Fourier transform has a prescribed decay as stated in  Theorem \ref{ingh-hei}. This proves the sufficiency part of the condition on the function $ \Theta $ appearing in the hypothesis. We then use this part of the  theorem to prove the necessity of the condition on $ \Theta.$ We begin with some preparations.

\subsection{Construction of $ F $} The Koranyi norm of  $x=(z,t)\in\mathbb{H}^n$  is defined  by $|x|=|(z,t)|=(|z|^4+t^2)^{\frac{1}{4}}.$ In what follows, we  work with the following left invariant metric defined by  $ d(x,y):=|x^{-1}y|,\ x,y\in\mathbb{H}^n.$ Given $a\in\mathbb{H}^n$ and $r>0$, the open ball of radius $r$ with centre at $a$ is defined by 
 	$$  B(a,r):=\{x\in\mathbb{H}^n:|a^{-1}x|<r\}.$$ 
With this definition, we note that if $f,g:\mathbb{H}^n\rightarrow\mathbb{C}$ are  such that $\text{supp}(f)\subset B(0,r_1)$ and $\text{supp}(g)\subset B(0,r_2)$, then we have 
 	$$   \text{supp}(f\ast g)\subset B(0,r_1).B(0,r_2)\subset B(0,r_1+r_2) , $$
	where  $f\ast g(x)=\int_{\mathbb{H}^n}f(xy^{-1})g(y)dy $ is the convolution of $ f $ with $ g.$

 	Suppose $\{\rho_j\}_j$ and $\{\tau_j\}_j$ are two sequences of positive real numbers such that both the series $\sum_{j=1}^\infty \rho_j$ and $\sum_{j=1}^\infty \tau_j$ are convergent.  We let $ B_{\C^n}(0,r) $ stand for the ball of radius $ r $ centered at $ 0 $ in $ \C^n$ and let $\chi_S$ denote the characteristic function of a set $S.$ For each $j\in\mathbb{N},$ we define 
 	functions $ f_j $ on $ \C^n $ and $ \tau_j $ on $ \R $ by
	$$f_j(z):=\rho_j^{-2n}\chi_{B_{\mathbb{C}^n}(0,a \rho_j)}(z),\,\,\,  ~ g_j(t):=\tau_j^{-2}\chi_{[-\tau_j^2/2,\tau_j^2/2]}(t)$$ where the positive constant $ a $ is chosen so that $\|f_j\|_{L^1(\mathbb{C}^n)}=1.$ We now  consider the functions $F_j:\mathbb{H}^n\rightarrow\mathbb{C}$ defined by 
 	$$F_j(z,t):=f_j(z)g_j(t),\ (z,t)\in\mathbb{H}^n.$$
In the following lemma, we record some useful, easy to prove,  properties of these functions.
 	\begin{lem}
 		Let $F_j$ be as above and  define $ G_N = F_1 \ast F_2\ast.....\ast F_N.$ Then we have  
 		\begin{enumerate}
 			\item $\|F_j\|_{L^{\infty}(\mathbb{H}^n)}\leq \rho_j^{-2n}\tau_j^{-2},$\,\, \,  $\|F_j\|_{L^1(\mathbb{H}^n)}=1,$
 			\item $\text{supp}(F_j)\subset B_{\mathbb{C}^n}(0,a \rho_j)\times [ -\tau_j^2/2,\tau_j^2/2]\subset B(0, a \rho_j+ c \tau_j) $, where $ 4 c^4 =1.$
 			\item  For any $N\in\mathbb{N}$, $\text{supp} (G_N) \subset B(0, a \sum_{j=1}^{N}\rho_j+ c  \sum_{j=1}^{N}\tau_j),\,\, \|G_N\|_1 =1.$
 			\item Given $x\in\mathbb{H}^n$ and $N\in\mathbb{N}$, $ F_2 \ast F_3 .....\ast F_N(x) \leq \rho^{-2n}_2\tau_2^{-2}.$ 
 		\end{enumerate} 
 	\end{lem}
 	We also recall a result about Hausd\"orff measure which will be used in the proof of the next theorem.  Let $\mathcal{H}^n(A)$ denote the $n$-dimensional Hausdorff measure of 
	$A\subset \mathbb{R}^n.$ Hausd\"orff measure coincides with the Lebesgue measure for Lebesgue measurable sets. For sets in $\mathbb{R}^n$ with sufficiently nice boundaries, the $(n-1)$-dimensional Hausdorff measure is same as the intuitive surface area. For more about this see   \cite[Chapter 7 ]{SS} .  Let $ A \Delta B $ stand for the symmetric difference between any two sets $ A $ and $ B.$ See \cite{ DS} for a proof of the following theorem.
 	\begin{thm}\label{hauss}
 		Let $ A \subset \mathbb{R}^n$ be a bounded set. Then for any $ \xi \in \R^n $
 		$$\mathcal{H}^n(A\Delta(A+\xi))\leq |\xi| \mathcal{H}^{n-1}(\partial A) $$
		where $ A+\xi$ is the translation of $ A $ by $ \xi$ and $ \partial A $ is the boundary of $ A.$
 	\end{thm}

 \begin{thm}\label{construct} The sequence defined by $ G_k  = F_1 \ast F_2\ast.....\ast F_k $ converges to a compactly supported  non-trivial $ F \in L^2(\He^n).$
 \end{thm}
 \begin{proof}
  In order show that $ (G_k) $ is Cauchy in $ L^2(\He^n) $ we first estimate $\| G_{k+1}-G_k \|_{L^{\infty}(\mathbb{H}^n)}.$ As  all the functions $ F_j $ have unit $ L^1 $ norm,  for any $x\in\mathbb{H}^n$ we have 
 		\begin{align*}
 		G_{k+1}(x)-G_k(x)&=\int_{\mathbb{H}^n} G_{k}(xy^{-1})F_{k+1}(y)dy-G_k(x) (x)\int_{\mathbb{H}^n}F_{k+1}(y)dy  \\
 		&=  \int_{\mathbb{H}^n}\left(G_k(xy^{-1})-G_k(x)\right)F_{k+1}(y)dy.
 		\end{align*}   
 As $ F_j $ are even we can change $ y $ into $ y^{-1}  $ in the above and   estimate the same as
 		\begin{equation}
 		\label{eq1}
 		|G_{k+1}(x)-G_{k}(x)|\leq \int_{\mathbb{H}^n}\left|G_k(xy)-G_k(x)\right|F_{k+1}(y)dy.
 		\end{equation} 
By defining $ H_{k-1} = F_2 \ast F_3......\ast F_k $ so that $ G_k = F_1 \ast H_{k-1} ,$ we have
$$ G_k(xy)-G_k(y) = \int_{\mathbb{H}^n}\left(F_1(xyu^{-1})-F_1(xu^{-1})\right) H_{k-1}(u) du$$
Using the estimate in Lemma 4.1 (4) we now estimate
 		\begin{equation}\label{eq2}
 		|G_k(xy)-G_k(x)|\leq \rho_2^{-2n}\tau_2^{-2}\int_{\mathbb{H}^n}\left|F_1(xyu^{-1})-F_1(xu^{-1})\right|du.
 		\end{equation}
 The change of variables $ u \rightarrow u x$  transforms the integral in the right hand side  of the above equation into
 		$$\int_{\mathbb{H}^n}\left|F_1(xyu^{-1})-F_1(xu^{-1})\right|du=\int_{\mathbb{H}^n}\left|F_1(xyx^{-1}u^{-1})-F_1( u^{-1})\right|du.$$  Since  the group $\mathbb{H}^n$ is unimodular, another change of variables $u \rightarrow u^{-1}$  yields
 		$$\int_{\mathbb{H}^n}\left|F_1(xyx^{-1}u^{-1})-F_1( u^{-1})\right|du=\int_{\mathbb{H}^n}\left|F_1(xyx^{-1}u )-F_1( u )\right| du.$$

Let $ x =(z,t) =(z,0)(0,t),~ y =(w,s)=(w,0)(0,s).$ As $ (0,t)$ and $ (0,s) $ belong to the center of $ \He^n $, an easy calculation shows that $ xyx^{-1} = (w,0) (0, s+\Im(z \cdot \bar{w})).$ With $ u=(\zeta,\tau) $ we have 
$$ xyx^{-1}u = (w+\zeta,0)(0, \tau+ s+ \Im(z \cdot \bar{w})-(1/2)\Im(\zeta\cdot \bar{w})).$$
Since $ F_1(z,t) = f_1(z)g_1(t) $ we see that the integrand $F_1(xyx^{-1}u )-F_1( u )$ in the above integral takes the form
$$  f_1(w+\zeta) g_1(\tau+ s+ \Im(z \cdot \bar{w})-(1/2)\Im(\zeta\cdot \bar{w}))- f_1(\zeta)g_1(\tau).$$
By setting $ b =  b(s, z,w,\zeta) =s+ \Im(z \cdot \bar{w})-(1/2) \Im(\zeta\cdot \bar{w}) $ we can rewrite the above as
\begin{equation}\label{eq3} \big( f_1(w+\zeta)-f_1(\zeta)\big) g_1(\tau+b)+ f_1(\zeta)\big( g_1(\tau+ b)-g_1(\tau)\big).\end{equation}

In order to estimate the contribution of the second term to  the integral under consideration
we first estimate the $ \tau$ integral as follows:
$$  \int_{-\infty}^\infty | g_1(\tau+b)-g_1(\tau)|  d\tau =  \tau_1^{-2} | (-b+K_\tau) \Delta K_\tau |$$
where $ K_\tau = [-\frac{1}{2}\tau_1^2, \frac{1}{2}\tau^2] $  is the support of $ g_1.$  For $ \zeta $ in the support of $ f_1 ,$ we have $ |\zeta| \leq a\rho_1 $ and hence
$$  | (-b+K_\tau) \Delta K_\tau | \leq 2 |b(z,w,\zeta)| \leq ( 2|s|+ |z||w|+ a\rho_1 |w|).   $$
Thus we have proved the estimate 
\begin{equation}\label{est-one}
\int_{\He^n}  f_1(\zeta)  |  g_1(\tau+ b)-g_1(\tau)| d\zeta d\tau \leq C \big(2|s|+(a\rho_1+|z|)|w| \big)\end{equation}
As $ g_1 $ integrates  to one, the contribution of the first term in (\ref{eq3}) is given by
$$  \int_{\C^n} | f_1(w+\zeta)-f_1(\zeta)| d\zeta   =  \rho_1^{-2n} \mathcal{H}^{2n}\left((-w+ B_{\mathbb{C}^n}(0, a\rho_1))\Delta B_{\mathbb{C}^n}(0,a\rho_1) \right) .$$
By appealing to Theorem \ref{hauss} in estimating the above, we obtain 
\begin{equation}\label{est-two}
\int_{\He^n} | f_1(w+\zeta)-f_1(\zeta)| \, g( \tau+b) \, d\zeta d\tau  \leq C |w| .
\end{equation}
Using the estimates (\ref{est-one}) and (\ref{est-two}) in (\ref{eq2}) we obtain
$$ |G_k(xy)-G_k(x)|\leq  C \rho_2^{-2n}\tau_2^{-2} \big( |s|+ (c_1+ c_2|z|)|w|)\big).$$ 
This estimate, when used in (\ref{eq1}), in turn gives us 
\begin{equation}\label{est-three}
|G_{k+1}(z,t)-G_{k}(z,t)|\leq C \int_{\mathbb{H}^n} \big( |s|+ (c_1+ c_2|z|)|w|)\big) F_{k+1}(w,s)\, dw\,ds
\end{equation}
where the constants $ c_1, c_2 $ and $ C $ depend only on $ n.$ Recalling that on the support of $ F_{k+1}(w,s) = f_{k+1}(w)g_{k+1}(s) $,  $ |w| \leq \rho_{k+1} $ and $ |s| \leq \tau_{k+1}^2 $, the above yields the estimate
\begin{equation}\label{est-four}
|G_{k+1}(z,t)-G_{k}(z,t)|\leq C \big( \tau_{k+1}^2 + (c_1+c_2 |z|) \rho_{k+1} \big).
\end{equation}
It is easily seen that the support of $ G_{k+1}-G_k $ is contained in $ B(0, a\rho+c\tau) $ where $ \rho = \sum_{j=1}^\infty \rho_j $ and $ \tau = \sum_{\tau_j}.$ Consequently, from the above we conclude that
$$  \|G_{k+1}-G_k\|_2 \leq  \| G_{k+1} - G_k \|_{\infty} \big(| B(0, a\rho+c\tau)|\big)^{1/2}  \leq C \big( \tau_{k+1}^2 + c_3 \rho_{k+1} \big).$$
From the above, it is clear that $ G_k $ is Cauchy in $ L^2(\He^n) $ and hence converges to a function $ F \in L^2(\He^n) $ whose support  is contained in $ B(0,a\rho+c\tau).$  The same argument shows that $ G_k $ converges to $ F $ in $ L^1.$ As $ \|G_k\|_1 =1 $ for any $ k,$ it follows that $ \|F||_1 =1 $ and hence $ F $ is nontrivial.
\end{proof}

\subsection{Estimating the Fourier transform of $ F$}
Suppose now that $\Theta $ is an even, decreasing function on $ \R $ for which $ \int_1^\infty \Theta(t) t^{-1} dt < \infty.$ We want to  choose two sequences $ \rho_j $ and $ \tau_j $  in terms of $ \Theta $ so that the series $ \sum_{j=1}^\infty \rho_j $ and $ \sum_{j=1}^\infty \tau_j $ both converge. We can then construct a function $ F $ as in Theorem \ref{construct} which will be compactly supported. Having done the construction we now want to compute the Fourier transform of the constructed function $ F $ and compare it with $ e^{-  \Theta(\sqrt{H(\lambda)}) \sqrt{H(\lambda)}}.$ This can be achieved by a judicious choice of the sequences $ \rho_j $ and $ \tau_j.$ As $ \Theta $ is given to be decreasing it follows that $ \sum_{j=1}^\infty  \frac{\Theta(j)}{ j} < \infty.$ It is then possible to  choose a decreasing sequence $ \rho_j $ such that $ \rho_j \geq  c_n^2 e^2  \frac{\Theta(j)}{j} $ (for a constant $ c_n $ to be chosen later) and $ \sum_{j=1}^\infty \rho_j < \infty.$ Similarly, we choose another decreasing sequence $ \tau_j $ such that $ \sum_{j=1}^\infty \tau_j <\infty.$

In the proof of the following theorem we require good estimates for the Laguerre coefficients of the function $ f_j(z) =\rho_j^{-2n}\chi_{B_{\mathbb{C}^n}(0,a \rho_j)}(z) $ where $ a $ chosen so that $ \|f_j\|_1 =1.$ These coefficients are defined by
\begin{equation}\label{lag-co} R_k^{n-1}(\lambda, f_j) =  \frac{   k! (n-1)!}{(k+n-1)!}  \int_{\C^n}  f_j(z) \varphi_{k,\lambda}^{n-1}(z) dz .\end{equation}

\begin{lem}\label{est-four-lem}  There exists a constant $ c_n > 0 $ such that 
$$  |R_k^{n-1}(\lambda,f_j)| \leq c_n \big(\rho_j \sqrt{(2k+n)|\lambda|}\big)^{-n+1/2}.$$
\end{lem}
\begin{proof} By abuse of notation we write $ \varphi_{k,\lambda}^{n-1}(r)$ in place of $\varphi_{k,\lambda}^{n-1}(z)$ when $ |z| =r.$ As $ f_j $ is defined as the dilation of a radial function,  the Laguerre coefficients are given by the integral
\begin{equation}\label{lag-co1} R_k^{n-1}(\lambda, f_j) =  \frac{2 \pi^n}{\Gamma(n)} \frac{   k! (n-1)!}{(k+n-1)!}  \int_0^a   \varphi_{k,\lambda}^{n-1}(\rho_j r) r^{2n-1} dr .\end{equation}  
When $ a \leq (\rho_j \sqrt{ (2k+n)|\lambda| })^{-1} $ we use the bound $\frac{   k! (n-1)!}{(k+n-1)!} |\varphi_{k,\lambda}^{n-1}(r)| \leq 1$ to  estimate  
$$          \frac{2 \pi^n}{\Gamma(n)}  \frac{   k! (n-1)!}{(k+n-1)!}  \int_0^a   \varphi_{k,\lambda}^{n-1}(\rho_j r) r^{2n-1} dr   \leq  \frac{ \pi^n a^{n+1/2}}{\Gamma(n+1)} 
\big(\rho_j \sqrt{(2k+n)|\lambda|}\big)^{-n+1/2}.$$ 
When $ a > (\rho_j \sqrt{ (2k+n)|\lambda| })^{-1} $ we split the integral into two parts, one of which gives the same estimate as above. To estimate the integral taken over $ (\rho_j \sqrt{ (2k+n)|\lambda| })^{-1} < r < a ,$ we use the bound  stated in Lemma  \ref{lem:T} which leads to the estimate 
$$          \frac{2 \pi^n}{\Gamma(n)}  \frac{   k! (n-1)!}{(k+n-1)!}  \int_{(\rho_j \sqrt{ (2k+n)|\lambda| })^{-1}}^a   \varphi_{k,\lambda}^{n-1}(\rho_j r) r^{2n-1} dr $$  
$$  \leq  C_n  \big(\rho_j \sqrt{(2k+n)|\lambda|}\big)^{-n+1/2} \int_0^a r^{n-1/2} dr = C_n^\prime a^{n+1/2}  \big(\rho_j \sqrt{(2k+n)|\lambda|}\big)^{-n+1/2}.$$
Combining the two estimates we get the lemma.
\end{proof}

\begin{thm}\label{compute}
 		Let $\Theta:\R\rightarrow [0,\infty)$ be an even, decreasing function with $\lim_{\lambda \rightarrow \infty}\Theta(\lambda)=0$  for which $\int_{1}^{\infty}\frac{\Theta(\lambda)}{\lambda} d\lambda <\infty.$ Let $ \rho_j $ and $ \tau_j $ be chosen as above. Then the Fourier transform of the function $ F $ constructed in Theorem \ref{construct} satisfies the estimate
 		\begin{align*}
 		\widehat{F}(\lambda)^*\widehat{F}(\lambda)\leq e^{-2\Theta (\sqrt{H(\lambda)}) \sqrt{H(\lambda)}}, \ \lambda\neq0.
 		\end{align*}
 		
 \end{thm}

 	\begin{proof} Observe that  $ F $ is radial  since each $ F_j $ is radial and hence the Fourier transform $ \widehat{F}(\lambda) $ is a function of the Hermite opertaor $ H(\lambda).$ More precisely, 
	\begin{align}
 		\widehat{F}(\lambda)= \sum_{k=0}^{\infty}R_k^{n-1}(\lambda, F) P_k(\lambda)
 		\end{align}
where  the Laguerre coefficients are  explicitly given by  (see  (2.4.7) in \cite{TH3}. There is a typo- the factor $ |\lambda|^{n/2}$ should not be there)
$$ R_k^{n-1}(\lambda, F) =  \frac{   k! (n-1)!}{(k+n-1)!}  \int_{\C^n}  F^\lambda(z) \varphi_{k,\lambda}^{n-1}(z) dz .$$
In the above, $ F^\lambda(z) $ stands for the inverse Fourier transform of $ F(z,t) $ in the $ t $ variable.
Expanding any $ \varphi \in L^2(\R^n) $ in terms of $ \Phi_\alpha^\lambda$ it is easy to see that the conclusion $\widehat{F}(\lambda)^*\widehat{F}(\lambda)\leq e^{-2\Theta (\sqrt{H(\lambda)}) \sqrt{H(\lambda)}} $ follows once we show that 
$$ (R_k^{n-1}(\lambda, F))^2 \leq C e^{-2\Theta (\sqrt{(2k+n)|\lambda}) \sqrt{(2k+n)|\lambda|}}   $$
for all $ k \in \mathbb{N} $ and $ \lambda \in \R^\ast.$ Now note that, by definition of $g_j$ and the choice of $a,$ we have 
 		$$|\widehat{g}_j(\lambda)|=\left|\frac{\sin (\frac{1}{2}\tau_j^2\lambda)}{\frac{1}{2}\tau_j^2\lambda}\right|\leq 1, \,\,\,  |R_k^{n-1}(\lambda, f_j)| \leq 1.$$ 
The bound on $ R_k^{n-1}(\lambda, f_j) $  follows from the fact that  $ |\varphi_k^\lambda(z)| \leq  \frac{(k+n-1)!}{k!(n-1)!} .$ Since $ F $ is constructed as the $ L^2 $ limit of the $ N$-fold convolution $ G_N = F_1 \ast F_2 ......\ast F_N $ we observe that for any $ N $
$$ ( R_k^{n-1}(\lambda, F))^2 \leq (R_k^{n-1}(\lambda, G_N))^2   =  (  \Pi_{j=1}^N R_k^{n-1}(\lambda, F_j))^2 $$
and hence it is enough to show that  for a given $ k $ and $ \lambda$ one can choose $ N=N(k,\lambda) $ in such a way that 
\begin{equation}\label{eq6} (  \Pi_{j=1}^N R_k^{n-1}(\lambda, F_j))^2 \leq C e^{-2\Theta (\sqrt{(2k+n)|\lambda|}) \sqrt{(2k+n)|\lambda|}} .\end{equation}
where $ C $ is independent of $ N.$   From the definition of $ G_N $ it follows that 
$$ \widehat{G_N}(\lambda) = \Pi_{j=1}^N \widehat{F_j}(\lambda)  =   \Pi_{j=1}^N \big( \sum_{k=0}^\infty  R_k^{n-1}(\lambda, F_j) P_k(\lambda) \big)$$
and hence  $ R_k^{n-1}(\lambda, G_N) =  \Pi_{j=1}^N R_k^{n-1}(\lambda, F_j).$ As $ F_j(z,t)  = f_j(z) g_j(t) $, we have
$$   R_k^{n-1}(\lambda, G_N) = \big( \Pi_{j=1}^N  \widehat{g_j}(\lambda) \big)  \big( \Pi_{j=1}^N R_k^{n-1}(\lambda, f_j) \big)  .                   $$
 As the first factor  is bounded by one, it is enough to consider the product $\Pi_{j=1}^N R_k^{n-1}(\lambda, f_j).$
 
 We now  choose $ \rho_j $ satisfying  $ \rho_j \geq  c_n^2\, e^2  \frac{\Theta(j)}{j} $  where $ c_n $ is the same constant appearing in Lemma \ref{est-four-lem}. We then take
 $ N=  \lfloor \Theta(((2k+n)|\lambda|)^{\frac{1}{2}})((2k+n)|\lambda|)^{\frac{1}{2}}\rfloor $ and consider
 $$\Pi_{j=1}^N R_k^{n-1}(\lambda, f_j) \leq \Pi_{j=1}^N c_n (\rho_j \sqrt{(2k+n)|\lambda|})^{-n+1/2} $$
 where we have used the estimates proved in Lemma \ref{est-four-lem}. As $ \rho_j $ is decreasing
 \begin{equation}\label{eq5} \Pi_{j=1}^N c_n (\rho_j \sqrt{(2k+n)|\lambda|})^{-n+1/2} \leq c_n^N    \big(\rho_N\sqrt{(2k+n)|\lambda| }\big)^{-(n-1/2)N}.\end{equation}
 By the choice of $ \rho_j $ it follows that 
 $$\rho_N^2(2k+n)|\lambda| \geq  c_n^4 e^4 \frac{\Theta(N)^2}{N^2} (2k+n)|\lambda| . $$
 As $ \Theta $ is decreasing and $ N \leq \sqrt{(2k+n)|\lambda|)} $ we have $ \Theta(N) \geq \Theta( \sqrt{(2k+n)|\lambda|}) $ and so 
 $$  \Theta(N)^2 (2k+n)|\lambda|  \geq \Theta\big( \sqrt{(2k+n)|\lambda|} \big)^2 (2k+n)|\lambda|  \geq N^2 $$
 which proves that $\rho_N^2(2k+n)|\lambda| \geq  c_n^4 e^4 .$ Using this in (\ref{eq5}) we obtain
 $$  \Pi_{j=1}^N c_n \big(\rho_j \sqrt{(2k+n)|\lambda|}\big)^{-n+1/2} \leq  (c_n^2 e^2)^{-(n-1)N} e^{-N}.    $$
 Finally, as $ N+1 \geq \Theta(((2k+n)|\lambda|)^{\frac{1}{2}})((2k+n)|\lambda|)^{\frac{1}{2}} $, we obtain the estimate (\ref{eq6}).
 \end{proof}
 
 \subsection{Ingham's theorem} We can now prove Theorem \ref{ingh-hei}. Since half of the theorem has been already proved, we only need to prove the following.
 
 \begin{thm}
 		Let $\Theta :\R\rightarrow [0,\infty)$ be an even, decreasing function with $\lim_{|\lambda|\rightarrow\infty}\Theta(\lambda)=0$ and $I=\int_{1}^{\infty} \Theta(\lambda) \lambda^{-1} d\lambda = \infty.$ Suppose the Fourier transform of $f\in L^1(\mathbb{H}^n)$  satisfies $$\hat{f}(\lambda)^*\hat{f}(\lambda)\leq e^{-\Theta(\sqrt{H(\lambda)})\sqrt{H(\lambda)}},\ \lambda\neq 0.$$ Further assume that $ f(\rho \, \omega \sqrt{\sin \theta}, \rho^2 \cos \theta) = f(\rho \, \omega , 0) $ in the Heisenberg coordinates.  If $f$ vanishes on a non-empty open set containing $0$  then $f=0\ a.e.$
 	\end{thm}
 	\begin{proof}Without loss of generality we can assume that $f$ vanishes on $B(0,\delta) .$ 		First we assume that $\Theta (\lambda)\geq 2|\lambda|^{-\frac{1}{2}},\,\, |\lambda| \geq 1.$
 		In view of Plancherel theorem for the group Fourier transform on the Heisenberg group we have $$\|\mathcal{L}^mf\|^2_2=(2\pi)^{-(n+1)}\int_{-\infty}^{\infty}\|\hat{f}(\lambda)H(\lambda)^m\|^2_{HS}|\lambda|^nd\lambda.$$ Using the formula for Hilbert-Schmidt norm of an operator we have 
 		$$\|\mathcal{L}^mf\|^2_2=(2\pi)^{-(n+1)}\int_{-\infty}^{\infty}\sum_{\alpha}((2|\alpha|+n)|\lambda|)^{2m}\|\hat{f}(\lambda)\Phi^{\lambda}_{\alpha}\|^2_2|\lambda|^nd\lambda$$ Now the given condition on the Fourier transform leads to the estimate
 		\begin{align*}
 		\|\mathcal{L}^mf\|^2_2\leq& C \int_{-\infty}^{\infty}\sum_{\alpha}((2|\alpha|+n)|\lambda|)^{2m}e^{-\Theta (((2|\alpha|+n)|\lambda|)^{\frac{1}{2}})((2|\alpha|+n)|\lambda|)^{\frac{1}{2}}}|\lambda|^nd\lambda\\
 		\leq & C\sum_{k=0}^{\infty}(2k+n)^{n-1}\int_{-\infty}^{\infty}((2k+n)|\lambda|)^{2m}e^{-\Theta(((2k+n)|\lambda|)^{\frac{1}{2}})((2k+n)|\lambda|)^{\frac{1}{2}}}|\lambda|^nd\lambda
 		\end{align*}
 		Now changing the variable from $\lambda$ to $(2k+n)^{-1}\lambda$ we get
 		$$\|\mathcal{L}^mf\|^2_2\leq C \sum_{k=0}^{\infty}(2k+n)^{-2}\int_{0}^{\infty} \lambda^{2m+n}e^{-\Theta(\sqrt{\lambda})\sqrt{\lambda}}d\lambda.$$ 		
The integral $ I $ appearing above can be estimated as follows. Under the extra assumption $\Theta (\lambda)\geq 2|\lambda|^{-\frac{1}{2}} $, on $ \Theta $ we have
\begin{align*} I = & \int_{0}^{m^8}\lambda^{2m+n}e^{-\Theta(\sqrt{\lambda})\sqrt{\lambda}}d\lambda+\int_{m^8}^{\infty}\lambda^{2m+n}e^{-\Theta(\sqrt{\lambda})\sqrt{\lambda}}d\lambda\\
 		\leq &\  2 m^{8(n+1)}\int_{0}^{m^4}\lambda^{4m-1}e^{-\Theta (m^4)\lambda}d\lambda + 4 \int_{m^2}^{\infty}\lambda^{8m+4(n+1)-1}e^{-2\lambda} d\lambda.
 		\end{align*} 
The above is dominated by a sum of two gamma integrals which can be evaluated to get
$$  I \leq 2 m^{8(n+1)} \Gamma(4m) \Theta(m^4)^{-4m} + 4 e^{-m^2} \Gamma(8m+4(n+1)).$$		
Using  Stirling's formula (see Ahlfors \cite{A}) $ \Gamma(x) = \sqrt{2\pi} x^{x-1/2} e^{-x} e^{\theta(x)/12x} ,  0< \theta(x) <1 $ valid for $ x > 0 ,$ we observe the the second term 	in the estimate for $ I $ goes to zero as $ m $ tends to infinity and the first term (and hence $ I $ itself ) is bounded by $ C (4m)^{4m} \Theta(m^4)^{-4m}.$

Thus we have proved the estimate 
$ \| \mathcal{L}^m f\|_2^2 \leq C (4m)^{4m} \Theta(m^4)^{-4m}.$
 The hypothesis on $ \Theta $ namely,  $\int_{1}^{\infty}\frac{\Theta (t)}{t}dt=\infty,$  by a change of variable implies that $\int_{1}^{\infty}\frac{\Theta (y^4)}{y}dy=\infty.$ Hence by integral test we get $\sum_{m=1}^{\infty}\frac{\Theta(m^4)}{m}=\infty.$ Therefore, it follows that  $\sum_{m=1}^{\infty}\|\mathcal{L}^mf\|^{-\frac{1}{2m}}_2=\infty.$ Since it vanishes on $B(0,\delta)$, $f$ and all its partial derivatives vanish at the origin. Therefore, by Chernoff's theorem for the sublaplacian we conclude that $f=0.$ Now we consider the general case. 
 
The function $\Psi(y)= (1+|y|)^{-1/2} $ satisfies  $\int_{1}^{\infty}\frac{\Psi(y)}{y}dy<\infty.$ By Theorem \ref{construct}  we can construct a radial function $ F \in L^2(\mathbb{H}^n)$ supported in $ B(0,\delta/2) $ such that $$\hat{F}(\lambda)^\ast \hat{F}(\lambda)\leq e^{-\Psi (\sqrt{H(\lambda)})\sqrt{H(\lambda)}},\ \lambda\neq 0.$$  As $ f $ is assumed  to vanish on $ B(0,\delta) ,$ the function $h=f\ast F$ vanishes on the smaller ball $ B(0,\delta/2).$ This can be easily verified by looking at 
 $$  f\ast F(x)=\int_{\mathbb{H}^n}f(xy^{-1})F(y)dy=\int_{B(0,\frac{\delta}{2})}f(xy^{-1})F(y)dy.$$ 
  When both $ x, y \in B(0,\delta/2),  \, d(0,xy^{-1})=|xy^{-1}|\leq |x|+|y|<\delta$ and hence $ f(xy^{-1}) =0$  proving that $ f\ast F(x) =0.$  The same is true for all the derivatives of $ h.$ We now claim that 
  $$  \widehat{h}(\lambda)^\ast \widehat{h}(\lambda) \leq e^{- 2 \Phi(\sqrt{H(\lambda)}) \sqrt{H(\lambda)}} $$
  where $ \Phi(y) = \Theta(y)+\Psi(y).$ As $ \widehat{h}(\lambda) = \widehat{f}(\lambda)\widehat{F}(\lambda) $, for any $ \varphi \in L^2(\R^n) $ we have
  $$ \langle \widehat{h}(\lambda)^\ast \widehat{h}(\lambda) \varphi, \varphi \rangle = \langle \widehat{f}(\lambda)^\ast \widehat{f}(\lambda) \widehat{F}(\lambda) \varphi, \widehat{F}(\lambda)\varphi \rangle .$$
 The hypothesis on $ f $ gives us the estimate 
 $$ \langle \widehat{f}(\lambda)^\ast \widehat{f}(\lambda) \widehat{F}(\lambda) \varphi, \widehat{F}(\lambda)\varphi \rangle  \leq C \langle e^{- 2 \Theta(\sqrt{H(\lambda)}) \sqrt{H(\lambda)}} \widehat{F}(\lambda) \varphi, \widehat{F}(\lambda)\varphi \rangle .$$
 As $ F $ is radial, $ \widehat{F}(\lambda) $ commutes with any function of $ H(\lambda)$ and hence the right hand side can be estimated using the decay of $ \widehat{F}(\lambda)$:
$$ \langle  \widehat{F}(\lambda)^\ast \widehat{F}(\lambda) e^{- \Theta(\sqrt{H(\lambda)}) \sqrt{H(\lambda)}}\varphi,  e^{- \Theta(\sqrt{H(\lambda)}) \sqrt{H(\lambda)}}\varphi \rangle 
  \leq C  \langle   e^{-2 (\Theta+\Psi)(\sqrt{H(\lambda)}) \sqrt{H(\lambda)}}\varphi,  \varphi \rangle .$$
  This proves our claim on $ \widehat{h}(\lambda) $ with $ \Phi= \Theta + \Psi.$  We also note that  $ \Phi(y) \geq |y|^{-1/2} .$  However, we cannot appeal to  Theorem \ref{CH}  since $ h = f \ast F $ need not satisfy the extra assumption $ h(\rho \, \omega \sqrt{\sin \theta}, \rho^2 \cos \theta) = h(\rho \, \omega , 0) .$ Nevertheless, we can modify the proof of Proposition \ref{non-essen} to complete the proof of Ingham's theorem.

Given $ F $ as above, let us consider $  \delta_rF(z,t) = F(rz,r^2t).$ It has been shown elsewhere (see e.g. \cite{LT}) that
$$ \widehat{\delta_rF}(\lambda) = r^{-(2n+2)}  d_r \circ \widehat{F}(r^{-2} \lambda) \circ d_r^{-1} $$
where $ d_r $ is the standard dilation on $ \R^n$ given by $ d_r\varphi(x) = \varphi(rx).$ The property of the function $ F ,$ namely $\hat{F}(\lambda)^\ast \hat{F}(\lambda)\leq e^{-2\Psi (\sqrt{H(\lambda)})\sqrt{H(\lambda)}}$ gives us 
$$ \widehat{\delta_rF}(\lambda)^\ast \widehat{\delta_rF}(\lambda) \leq C r^{-2(2n+2)}    d_r \circ e^{-2\Psi (\sqrt{H(\lambda/r^2)})\sqrt{H(\lambda/r^2)}} \circ d_r^{-1} .$$
Testing against $ \Phi_\alpha^\lambda $ we can simplify the right hand side which gives us
$$ \widehat{\delta_rF}(\lambda)^\ast \widehat{\delta_rF}(\lambda) \leq C r^{-2(2n+2)}      e^{-2\Psi_r (\sqrt{H(\lambda)})\sqrt{H(\lambda)}} $$ 
where $ \Psi_r(y) = \frac{1}{r}\Psi(y/r).$ If we let $ F_\varepsilon(x) = \varepsilon^{-(2n+2)} \delta_{\varepsilon}^{-1}F(x) $ then it follows that $F_\varepsilon $ is an approximate identity. Moreover, $ F_\varepsilon $ is supported in $ B(0,\varepsilon \delta/2) $ and satisfies the same hypothesis as $ F $ with $\Psi(y)$ replaced by $ \varepsilon \Psi(\varepsilon y)$ which has the same integrability and decay conditions. 

 Returning to the proof of Proposition \ref{non-essen} we only need to produce a subset $ V  \subset D_{St}(A)$  satisfying the conditions $(i) $ and $ (ii).$ Then the linear span of $ D_{St}(A) $ will be dense contradicting the already proved result that $ A $ is not essentially self-adjoint. Suppose the function $ f $ given in Theorem 4.6 is non trivial. We consider the set
$$  V = \{  R_\sigma \circ \delta_r (f \ast F_\varepsilon) :  r, \varepsilon > 0,  \sigma \in U(n) \} .$$ 
As $ F_\varepsilon $ satisfies the same hypothesis as $ F $ with $ \Psi $ replaced by $ \varepsilon \Psi(\varepsilon y)$ it follows that $ f \ast F_\varepsilon $ satisfies the hypothesis of the theorem with $ \Theta(y) $ replaced by $ \Theta(y) +\varepsilon \Psi(\varepsilon y)$. By the previous part of the theorem it follows that   $\sum_{m=1}^{\infty}\|\mathcal{L}^m(f \ast F_\varepsilon)\|^{-\frac{1}{2m}}_2=\infty.$ Hence for any $ 0 < \varepsilon < 1, f \ast F_\varepsilon \in D_{St}(A) $ and consequently $ V \subset D_{St}(A).$  Thus $ V $ satisfies the condition $(i).$

 To verify the condition $ (ii) $ we observe that $ R_\sigma( f\ast g) = R_\sigma f \ast g $ for any radial $ g $ and $ \delta_r (f \ast g) = r^{2n+2} \, (\delta_r f \ast \delta_r g ) $. Consequently, as $ F $ is radial it follows that 
$ R_\sigma \circ \delta_r (f \ast F_\varepsilon) = (R_\sigma \circ \delta_r f ) \ast F_{\varepsilon/r}.$ Since $ f(\rho \, \omega \sqrt{\sin \theta}, \rho^2 \cos \theta) = f(\rho \, \omega , 0) ,$ we can proceed as in the proof of Theorem  \ref{CH} to show that for any $ (z,t) $ we can find $ r $ and $ \sigma $ such that $ R_\sigma \circ \delta_r f (z,t) \neq 0.$ As $ f $ is smooth and $ F_\varepsilon $ is an approximate identity, $(R_\sigma \circ \delta_r f ) \ast F_{\varepsilon/r}(z,t)$ converges to $(R_\sigma \circ \delta_r f )(z,t) $ as $ \varepsilon  \rightarrow 0.$ Hence  for all small enough $ \varepsilon $ we have $(R_\sigma \circ \delta_r f ) \ast F_{\varepsilon/r}(z,t) \neq 0.$ This proves our claim that $ V $ satisfies condition $(ii) $ and completes the proof.
\end{proof}

\section*{Acknowledgments}
The authors would like to thank Prof. Swagato K. Ray for his suggestions and discussions.
The first author is supported by Inspire faculty award from D. S. T, Govt. of India.
The second author is supported by Int. Ph.D. scholarship from Indian Institute of Science. The third author is supported by a research fellowship from Indian Statistical Institute.
The last author is supported by  J. C. Bose Fellowship from D.S.T., Govt. of India  as well as a grant from U.G.C.\\


\begin{thebibliography}{99} 



\bibitem{A} Lars V. Ahlfors, \emph{Complex Analysis:} {An introduction to the theory of analytic functions of one complex variable}, 3rd Edition, McGraw-Hill, Inc. 

  

\bibitem{BRS}  M. Bhowmik,  S. K. Ray, and S.Sen, {Around theorems of Ingham-type regarding decay of
Fourier transform on $\mathbb{R}^n, \mathbb{T}^n$ and two step nilpotent Lie Groups},  \emph{Bull. Sci. Math}, \textbf{155} (2019) 33-73. 

\bibitem{BPR} M. Bhowmik, S. Pusti, and S. K. Ray, {Theorems of Ingham and Chernoff on Riemannian symmetric spaces of noncompact type}, \emph{Journal of Functional Analysis}, Volume 279, Issue 11 (2020).

\bibitem{BS}  M. Bhowmik and S. Sen, {Uncertainty principles of Ingham and Paley-wiener on semisimple Lie groups}, \emph{Israel Journal of Mathematics} \textbf{225} (2018),193-221.

\bibitem{BT} S. Bochner and A. E. Taylor, {Some theorems on quasi-analyticity for functions of several variables}, \emph{Amer. J. Math.},\textbf{61} (1939), no-2,303-329.


\bibitem{ Boc} S. Bochner,  { Quasi-analytic functions, Laplace operator, positive kernels} \emph{Ann. of Math. (2)}, \textbf{51} 
(1950), 68-91. MR0032708 (11,334g)


\bibitem{CH1} P. R. Chernoff, { Some remarks on quasi-analytic vectors}, \emph{ Trans. Amer. Math. Soc.}
\textbf{167} (1972), 105-113. 


\bibitem{CH2} P. R.Chernoff, {Quasi-analytic vectors and quasi-analytic functions.} \emph{Bull. Amer. Math. Soc.} \textbf{81}  (1975), 637-646.


\bibitem{F} G. B. Folland, \emph{Harmonic Analysis in Phase Space,} Ann. Math. Stud. \textbf{122}. Princeton University Press, Princeton,
N.J., 1989.

\bibitem{I}  A. E. Ingham,  { A Note on Fourier Transforms}, \emph{J. London Math. Soc.} S1-9 (1934), no. 1, 29-32.
MR1574706

\bibitem{KOS}  B. Kr\"otz, G. Olafasson and R. J. Stanton, {The image of heat kernel transform on Riemannian symmetric spaces of the noncompact type}, \emph{Int.Math.Res.Not.} (2005), no.22, 1307-1329.


\bibitem{KTX} B. Kr\"otz, S. Thangavelu and Y.Xu, {Heat kernel transform for nilmanifolds associated to the Heisenberg group}, \emph{Rev.Mat.Iberoam.} \textbf{24} (2008), no.1, 243-266. 

\bibitem{LT} R. Lakshmi Lavanya and S. Thangavelu, {Revisiting the Fourier transform on the Heisenberg group}, \emph{Publ. Mat.} \textbf{58} (2014), 47-63.

\bibitem{Levinson} N. Levinson, On a Class of Non-Vanishing Functions, \emph{Proc. London Math. Soc.}  \textbf{41} (1)(1936)  393-407. MR1576177

\bibitem{M} C. Markett, {Mean Ces\`aro summability of Laguerre expansions and norm estimates with shifted parameter}, \emph{Anal. Math.} \textbf{8} (1982), no. 1, 19--37.

\bibitem{Mu} B. Muckenhoupt, {Mean convergence of Hermite and Laguerre series. II} \emph{Trans. Amer. Math. Soc.} \textbf{147} (1970), 433--460


\bibitem{MM} D. Masson and W. McClary, {Classes of $C^{\infty}$ vectors and essential self-adjointness}, \emph{J. Functional Analysis}, \textbf{10} (1972), 19-32.

\bibitem{AN} A. E. Nussbaum, Quasi-analytic vectors, \emph{Ark. Mat.} \textbf{6} (1965), 179-191. MR 33
No.3105.
 
\bibitem{ANs} A. E. Nussbaum, {A note on quasi-analytic vectors,} \emph{Studia Math.} \textbf{33} (1969), 305-309. 
 
 
 \bibitem{PW1} R. E. A. C. Paley and  N.  Wiener,   Notes on the theory and application of Fourier transforms. I, II. \emph{Trans. Amer.
 Math. Soc.} \textbf{35} (1933), no. 2, 348-355.  
 
 \bibitem{PW2} R. E. A. C. Paley and  N.  Wiener, \emph{Fourier transforms in the complex domain} (Reprint of the 1934 original)
 American Mathematical Society Colloquium Publications, 19. American Mathematical Society, Providence,
 RI, 1987.  
 
 \bibitem{RB} M. Reed and B. Simon, \emph{Methods of modern mathematical physics I: Functional Analysis}, Academic Press, INC. (London) LTD.
 
 
 \bibitem{RT} L. Roncal and S. Thangavelu, {An extension problem and trace Hardy inequality for the sublaplacian on the $H$-type gropus},   \emph{Int.Math.Res.Not.} (2020) no.14, 4238-4294. 

\bibitem{WR} W. Rudin, Functional Analysis, 2nd edition,  McGraw-Hill, Inc.  

\bibitem{DS} D. Schymura, {An upper bound on the volume of the symmetric difference of a body and a congruent copy},  \emph{Adv.Geom.} \textbf{14} (2014) no. 2, 287-298.

\bibitem{BS} B. Simon, The theory of semi-analytic vectors: A new proof of a theorem of Masson
and McClary, \emph{Indiana Univ. Math. J.} 20 (1970/71), 1145-1151. MR 44 No.7357.


\bibitem{SS} E. M. Stein and R. Shakarchi, \emph{Real Analysis: Measure Thery, Integration and Hilbert spaces}, Princeton Lectures in Analysis III, Princeton University Press. 

\bibitem{Taylor} M. E. Taylor, \emph{Non-commutative harmonic analysis}, Amer. Math.Soc., Providence, RI, (1986).


\bibitem{TH1}   S. Thangavelu, \emph{Lectures on Hermite and Laguerre expansions}, Mathematical Notes {\bf 42}. Princeton University Press, Princeton, NJ, 1993.

\bibitem{TH2} S. Thangavelu, \emph{Harmonic Analysis on the Heisenberg group}, Progress in Mathematics {\bf 159}. Birkh\"auser, Boston, MA, 1998.

\bibitem{TH3}  S. Thangavelu, \emph{An introduction to the uncertainty principle. Hardy's theorem on Lie groups. With a foreword by
Gerald B. Folland}, Progress in Mathematics 217. Birkh\"auser, Boston, MA, 2004

\bibitem{TH5} S. Thangavelu, {Gutzmer's formula and Poisson integrals on the Heisenberg group}, \emph{Pacific J.Math.} \textbf{231} (2007), no.1, 217-237. 

\bibitem{TH4}  S. Thangavelu, {An analogue of Pfannschmidt's theorem for the Heisenberg group}, \emph{The Journal of Analysis}, \textbf{26} (2018) 235-244.


\end{thebibliography}
\end{document}